\documentclass{amsart}
\usepackage[utf8]{inputenc}
\usepackage{amsfonts}
\usepackage{amsmath}
\usepackage{amssymb}
\usepackage{amsthm}
\usepackage{xcolor}
\usepackage{bm}
\usepackage{mathtools}
\usepackage{enumitem}
\usepackage{verbatim}
\usepackage{pinlabel}
\usepackage{caption}
\usepackage{subcaption}
\usepackage{tikz}
\usepackage{graphicx}
\usepackage{hyperref}
\usepackage{tikz-cd}
\usepackage{biblatex}
\usepackage{float}
\usepackage{geometry}
\newtheorem{theorem}{Theorem}[section]

\newtheorem{lemma}[theorem]{Lemma}
\newtheorem{obsthm}[theorem]{Observation/Theorem}

\newtheorem{proposition}[theorem]{Proposition}

\newtheorem{corollary}[theorem]{Corollary}

\newtheoremstyle{claim}% name
  {\topsep}% space above
  {\topsep}% space below
  {}% body font
  {}% indent amount
  {\itshape}% theorem head font
  {.}% punctuation after theorem head
  {.5em}% space after theorem head
  {\thmname{#1}\thmnumber{ #2}\thmnote{ (#3)}}% theorem head spec
\theoremstyle{claim}

\DeclareMathOperator{\Homeo}{Homeo}
\DeclareMathOperator{\Diffeo}{Diff}

\DeclareMathOperator{\Aut}{Aut}
\DeclareMathOperator{\link}{link}
\DeclareMathOperator{\diam}{diam}

\DeclareMathOperator{\df}{d_0}
\DeclareMathOperator{\dof}{d_1}
\DeclareMathOperator{\dkf}{d_k}
\DeclareMathOperator{\dkof}{d_{k+1}}

\newcommand{\fine}{{\mathcal{C}^\dagger}}
\newcommand{\onefine}{{\mathcal{C}^\dagger_1}}
\newcommand{\kfine}{{\mathcal{C}^\dagger_k}}
\newcommand{\kpofine}{\mathcal{C}^\dagger_{k+1}}
\newcommand{\finfine}{{\mathcal{C}^\dagger_{<\infty}}}
\newcommand{\E}{\mathcal{E}}
\newcommand{\surf}{S_{g,b}}

\newcommand{\pit}[1]{\medskip\noindent\textit{#1}\textit{.}}

\newcommand{\p}[1]{\medskip\noindent\textbf{#1}\textbf{.}}

\title{Hyperbolicity, topology, and combinatorics of fine curve graphs and variants}
\author{Roberta Shapiro}
\email{shaprh@umich.edu}

\begin{document}

\maketitle
\begin{abstract}
    Given a surface, the fine $k$-curve graph of the surface is a graph whose vertices are simple closed essential curves and whose edges connect curves that intersect in at most $k$ points. We note that the fine $k$-curve graph is hyperbolic for all $k$ and, for $k\geq 2,$ show that it contains as induced subgraphs all countable graphs. We also show that the direct limit of this family of graphs, which we call the finitary curve graph, has diameter 2, has a contractible flag complex, contains every countable graph as an induced subgraph, and has as its automorphism group the homeomorphism group of the surface. Finally, we explore some finite graphs that are not induced subgraphs of fine curve graphs.
   \end{abstract}
 
\section{Introduction}

Let $S$ be an orientable surface, sometimes denoted $\surf$ (if $S$ has genus $g$ and $b$ boundary components) or just $S_g$ (if $S$ is closed and compact). The \textit{fine $k$-curve graph of $S$}, denoted $\kfine(S),$ is the graph whose vertices are embedded simple, closed, essential curves. Two vertices $u$ and $v$ are connected by an edge if $|u\cap v|\leq k.$ When $k=0$, we call the resulting graph the \emph{fine curve graph}, denoted $\fine(S)$. %Unless otherwise stated, we will assume that $S$ is closed, connected, orientable, and of genus at least $2.$ 

We would be amiss to not include the observation that $\kfine(S)$ is hyperbolic.

\begin{obsthm}\label{maintheoremkfine}
    Let $S_g$ be a closed, orientable surface with $g\geq 2.$ Then, $\kfine(S_g)$ is hyperbolic.
\end{obsthm}

Theorem~\ref{maintheoremkfine} is known for the fine curve graph due to work of Bowden--Hensel--Webb \cite{Bowden_Hensel_Webb_2021} and the proof reduces to this fact via quasi-isometries. 

We now define the \textit{finitary curve graph of $S$,} denoted $\finfine(S),$ to be
\[ \finfine(S):= \varinjlim_k \kfine(S).\] 
The vertices of $\finfine(S)$ are again embedded essential simple closed curves in $S,$ but now edges connect curves that intersect at finitely many points.

Our second  theorem computes the diameter of $\finfine(S)$ in the path metric.

\begin{theorem}\label{maintheoremfinitefine}
    Let $S=S_{g,b}$ be a compact, orientable surface with $g\geq 1$ or $b\geq 4.$ Then, $\diam(\finfine(S))=2$.
\end{theorem}

So $\finfine(S_g)$ is quasi-isometric to a point. Define the \textit{flag complex} of $\finfine(S_g)$ to be the simplicial complex where vertices $v_0,\ldots,v_k$ span a $k$-cell if $|v_i\cap v_j|<\infty$ for all $i\neq j.$ We may then ask whether the flag complex $\finfine(S_g)$ is homotopic to a point. This question is answered in the affirmative in the following theorem.

\begin{theorem}\label{maintheoremcontractibility}
    Let $\surf$ be an orientable surface with $g\geq 1$ or $b\geq 4.$ Then the flag complex of $\finfine(\surf)$ is contractible.
\end{theorem}

We also study more combinatorial properties of the fine curve graph and finitary curve graph. In the following theorems, we show that any finite graph can be embedded as an induced subgraph of the fine curve graph of a surface of high enough genus or the fine $k$-curve graph of any surface that contains an essential embedded annulus.

\begin{theorem}\label{maintheorem:finite.in.fine}
    Let $G$ be a finite graph on $n$ vertices. Then $G$ is isomorphic to an induced subgraph of $\fine(S_g)$ with $g\geq {n \choose 2}-n+1$, so $g=O(n^2).$
\end{theorem}

% \begin{theorem}\label{maintheorem.finite.in.finitary}
%     Let $G$ be a finite graph. Then $G$ is isomorphic to an induced subgraph of $\finfine(\surf)$ with $g\geq 1$ or $b\geq 4.$
% \end{theorem}

% Our final theorem relates to the Erd\H{o}s-R\'enyi graph, which is the unique (up to isomorphism) random graph on countably many vertices.

% \begin{theorem}\label{maintheorem.erg.in.finitary}
%     Let $ERG$ be the Erd\H{o}s-R\'enyi graph. Then $ERG$ is isomorphic to an induced subgraph of $\finfine(\surf)$ with $g\geq 1$ or $b\geq 4.$
% \end{theorem}

\begin{theorem}\label{maintheorem.countable.subgraphs.in.finitary}
    Let $G$ be a graph on countably many vertices and $\surf$ be an orientable surface with $g\geq 1$ or $b\geq 4.$ Then $G$ is isomorphic to an induced subgraph of $\finfine(\surf)$ and $\mathcal{C}_k^\dagger(\surf)$ for $k\geq 2$.
\end{theorem}

% As corollaries, we have the following.
% \begin{corollary}\label{cor:finite.in.finitary}
%     Let $G$ be a finite graph. Then $G$ is isomorphic to an induced subgraph of $\finfine(\surf)$ and $\mathcal{C}_k^\dagger(\surf)$ with $k\geq 2.$
% \end{corollary}

We have a corollary that relates to the Erd\H{o}s-R\'enyi graph, which is the unique (up to isomorphism) random graph on countably many vertices.

\begin{corollary}\label{cor.erg.in.finitary}
    Let $ERG$ be the Erd\H{o}s-R\'enyi graph. Then $ERG$ is isomorphic to an induced subgraph of $\finfine(\surf)$ and $\mathcal{C}_k^\dagger(\surf)$ with $k\geq 2.$
\end{corollary}

% {\color{red}{Questions:
% \begin{enumerate}
%     \item is the finitary curve graph self-complementary? (The Rado graph is) (complement ~= original graph)
%     \item is there a better lower bound for complexity of a surface s.t. FCG contains any subgraph on $n$ vertices?
%     \item Aut of single isotopy class FCG (maybe more natural with Ryan)
% \end{enumerate}
% }}

Define the \emph{curve graph} of a surface $\Sigma$ to be the graph whose vertices are isotopy classes of essential simple closed curves and whose edges connect pairs of isotopy classes that admit disjoint representatives. Bering--Gaster prove that the curve graph of a surface $\Sigma$ has $ERG$ as an induced subgraph if and only if $\Sigma$ has infinite genus \cite{bg}. One would hope that the fine curve graph would grant us more flexibility in this regard, but alas, the proof of Bering--Gaster holds. 

However, we have a new method of constructing finite graphs that, given a surface of genus $g$ with $b$ boundary components, there is a graph with $O(g+b)$ vertices does not appear as an induced subgraph. We call such graphs \emph{inadmissible}.

\begin{theorem}[Construction of inadmissible graphs]\label{maintheorem:forbiddensubgraphs}
    Let $\Sigma=S_{g,b}$. Then there exists a graph $G_{\Sigma}$ that does not appear as an induced subgraph of $\fine(\Sigma).$  
\end{theorem}

Finally, we find the automorphism group of the finitary curve graph in the following theorem.
\begin{theorem}\label{maintheorem:aut}
    Let $S_g$ be a closed oriented surface with $g\geq 1.$ Then the natural map \[\Phi: \Homeo(S_g)\to \Aut\finfine(S_g)\] is an isomorphism.
\end{theorem}

Along the way, we prove an incredibly useful result that is the hammer behind most of the above theorems: given a finite collection of arcs and curves in a surface, there exists an arc or curve in any isotopy class and arbitrarily close to any representative such that it intersects all curves or arcs in the collection finitely many times each and at crossing intersections (see Section~\ref{sec:kcurve} for definitions). This is the content of Lemma~\ref{lemma:finitelymanyintersections} and Corollary~\ref{cor:choosecurve}.

\p{A brief history of curve graphs} The \textit{curve graph} of a surface $S$ is the graph whose vertices are isotopy classes of essential simple closed curves and whose edges connect vertices that admit disjoint representatives. The curve graph is a classical tool used to study the \textit{mapping class group} of $S,$ which is the group of orientation-preserving homeomorphisms modulo isotopy. In a seminal paper, Masur--Minksy show that the curve graph is Gromov hyperbolic \cite{MM}. The $k$-curve graph, which has the same vertices as the curve graph while edges connect curves that minimally intersect at most $k$ times, is remarked to be hyperbolic by Agrawal--Aougab--Chandran--Loving--Oakley--Shapiro--Xiao \cite{kcurve}.

An object of recent study is the \textit{fine curve graph}, which is defined to be the fine $k$-curve graph with $k=0$ and is notated $\fine(S).$ It was originally introduced by Bowden--Hensel--Webb in the case that vertices are smooth curves to prove that $\Diffeo_0(S_g)$ admits unbounded quasimorphisms (for $g\geq 1$), and, as a corollary, is not uniformly perfect \cite{Bowden_Hensel_Webb_2021}. Fine curve graphs with edges not corresponding to disjointness were introduced in Le~Roux--Wolff \cite{LRW} and Booth--Minahan--Shapiro \cite{BMS}; both papers study automorphism groups of such graphs. Moreover, Booth studies the automorphism group of the fine curve graph whose vertices are $C^\infty$ curves and edges correspond to disjointness \cite{booth1} and also studies homeomorphisms that preserve the set of continuously differentiable curves \cite{booth2}.

\pit{A note on the tameness of curves} Any curve in a surface is tame: at each point, there is a neighborhood of the point where the curve is homeomorphic to one that is locally flat \cite[Theorem A1]{Epstein}. In other words, given a curve $c,$ there is a neighborhood of each point $x\in c$ that is hoemeomorphic to the open unit disk such that the image of $c$ under that homeomorphism is the $x$-axis. This further implies that any two nonseparating curves can be taken to each other by a homeomorphism of the surface.

Despite curves themselves being tame, any two curves can have incredibly complicated intersection patterns. For example, two curves can intersect in at isolated points, in intervals, or in Cantor sets (just to name a few). Moreover, given parameterized unit square $I\times I$ in $S$, there is a way to construct a curve in $S$ such that every arc $\{t\}\times I$ intersects this curve infinitely many times. As such, we will mainly use the fact that the intersection set of two curves is compact. 

\p{Outline} In Section~\ref{sec:kcurve}, we prove Theorem~\ref{maintheoremkfine} using surgery arguments, induction, and quasi-isometry. In Section~\ref{sec:finefinite}, we prove Theorem~\ref{maintheoremfinitefine}. In Section~\ref{sec:contractability}, we prove Theorem~\ref{maintheoremcontractibility}. In Section~\ref{sec:subgraphs}, we prove Theorems~\ref{maintheorem:finite.in.fine}, \ref{maintheorem.countable.subgraphs.in.finitary}, and \ref{maintheorem:forbiddensubgraphs}. Finally, in Section~\ref{sec:aut}, we prove Theorem~\ref{maintheorem:aut}.

\p{Acknowledgements} Thank you to Dan Margalit for his support; to Katherine Booth, Sean Eli, Jonah Gaster, Daniel Minahan, and especially Alex Nolte for many discussions; to Denis Osin for asking the questions that prompted this work; and to Edgar Bering IV, Sarah Koch and Alex Wright for comments on a draft.

\section{Hyperbolicity of fine $k$-curve graphs}\label{sec:kcurve}

The goal of this section is to prove Theorem~\ref{maintheoremkfine}. We will do this by induction on $k$ and show that the fine $k$-curve graph is quasi-isometric to the fine $(k-1)$-curve graph. The base case of the fine $1$-curve graph and the fine curve graph is addressed in Proposition~\ref{prop:onefineqifine}. The inductive step is proven in Proposition~\ref{prop:konefineqikfine}. 

We first introduce some nomenclature for points of $u\cap v$ that are not accumulation points for the set $u\cap v$. Two curves $u$ and $v$ are \textit{touching} at a point $c\in u\cap v$ if, in an arbitrarily small neighborhood $N$ of $c,$ $u$ and $v$ can be isotoped to be disjoint. Otherwise, $u$ and $v$ are \textit{crossing} at $c.$ Examples of such intersections can be found on the right and left, respectively, of Figure~\ref{fig:crossingvstouching}. 

\begin{figure}[h]
\begin{center}
\includegraphics[width=3in]{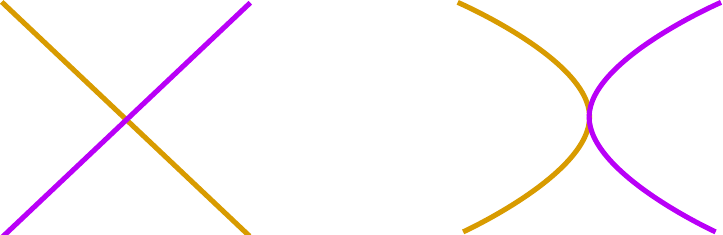}
\caption{Left: a crossing intersection. Right: a touching intersection.}\label{fig:crossingvstouching}
\end{center}
\end{figure}

We now introduce several graph notations. If $\{v_i\}_I$ are vertices of a graph $G$, we denote a (possibly infinite) path in $G$ by $(\ldots, v_{j-1},v_j,v_{j+1},\ldots).$ We endow all graphs with the path metric, wherein the distance between two vertices is the length of a shortest path between them. We further parameterize all edges to have unit speed and length one. When considering quasi-isometries, we can work solely with vertices, as every point on an edge is distance at most $\frac{1}{2}$ from a vertex.

\begin{proposition}[Base Case]\label{prop:onefineqifine}
Let $S_g$ be a closed, orientable surface with $g\geq 2.$ Then, $\fine(S_g)$ is quasi-isometric to $\onefine(S_g)$.
\end{proposition}

\begin{proof}
Let $\iota:\fine(S_g)\to\onefine(S_g)$ be the natural inclusion map. Let $\df(\cdot,\cdot)$ and $\dof(\cdot,\cdot)$ denote distances in $\fine(S_g)$ and $\onefine(S_g),$ respectively.

We will show that $\iota$ is a quasi-isometry with respect to $\df$ and $\dof.$ We will abuse notation and denote $\iota(u)$ by $u$. Quasi-surjectivity is achieved since $\iota$ is bijective on the vertices and every point on an edge is distance at most $\frac{1}{2}$ from any vertex. Thus it is enough to check the distance conditions required by quasi-isometry by considering vertices.

First, for any two vertices $u$ and $v,$ we have that $\dof(u,v)\leq \df(u,v),$ as the image of any path in $\fine(S_g)$ is a path of the same length in $\onefine(S_g).$  We now claim that $\df(u,v) \leq 2\dof(u,v)$. 

Let $u$ and $v$ be vertices such that $\dof(u,v)=1.$ If $u$ and $v$ are disjoint, we are done. If $u$ and $v$ are touching, isotope $u$ off of itself and away from $v$ to make $u'.$ We then have that $(u,u',v)$ is a path in $\fine(S_g),$ so $\df(u,v)=2\leq 2\dof(u,v).$ If $u$ and $v$ are crossing, then up to homeomorphism, $u$ and $v$ are the curves pictured in the left of Figure~\ref{fig:genustwowithcurves}. There is therefore a curve $w$ in the subsurface not filled by $u$ and $v.$ An example of such a curve can be seen on the right of Figure~\ref{fig:genustwowithcurves}.

\begin{figure}[h]
\begin{center}
\includegraphics[width=3in]{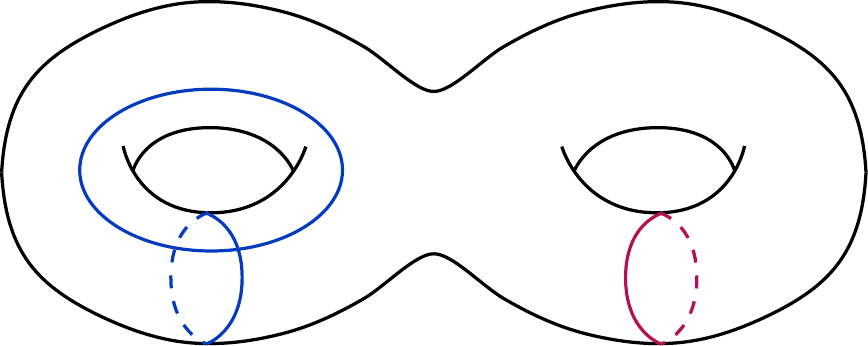}
\caption{Any pair of crossing curves in $\onefine(S_g)$ is, up to homeomorphism of $S_g$, equivalent to the two blue curves on the left. We may find a curve disjoint from both (such as the red curve on the right) outside the torus they fill.}\label{fig:genustwowithcurves}
\end{center}
\end{figure}

We therefore conclude that, for all vertices $u$ and $v$, we have $\frac{1}{2}\df(u,v)\leq\dof(\iota(u),\iota(v))\leq 2\df(u,v)$.
\end{proof}

With that in mind, we will now prove the inductive step.

\begin{proposition}[Inductive Case]\label{prop:konefineqikfine}
Let $S_g$ be a closed, orientable surface with $g\geq 2.$ Then $\kfine(S_g)$ is quasi-isometric to $\kpofine(S_g)$.
\end{proposition}

\begin{proof}
Let $\iota:\kfine(S_g)\to\kpofine(S_g)$ be the natural inclusion map. Let $\dkf(\cdot,\cdot)$ and $\dkof(\cdot,\cdot)$ denote distances in $\kfine(S_g)$ and $\kpofine(S_g)$ respectively. 

We will show that $\iota$ is a quasi-isometry with respect to $\dkf$ and $\dkof.$ We will abuse notation and denote $\iota(u)$ by $u$. Quasi-surjectivity follows as in the base case. It is enough to check the quasi-isometric embedding condition on vertices. 

First, for any two vertices $u$ and $v,$ we have that $\dkof(u,v)\leq \dkf(u,v),$ as the image of any path in $\kfine(S_g)$ is a path in $\kpofine(S_g).$ It remains to show that for any vertices $u$ and $v$, $\dkf(u,v)\leq 3\cdot \dkof(u,v).$

Suppose $u_i$ and $u_{i+1}$ are consecutive vertices along a geodesic path in $\kpofine(S_g)$ that intersect $k+1$ times. We claim that $d_k(u_i,u_{i+1})\leq 3.$ We prove this via casework.

\p{Case 1: Touching intersection} Suppose $u_i$ and $u_{i+1}$ have a touching intersection. Isotope $u_i$ in a neighborhood of this intersection to create $u'_i.$ We then have that $|u'_i\cap u_{i+1}|=k,$ so $u'_i$ and $u_{i+1}$ are adjacent in $\kfine(S_g).$ It remains to find a length 2 path between $u_i$ and $u'_i.$ 
    
Consider a regular neighborhood of $u_i\cup u'_i.$ Take $u''_i$ to be a boundary component of this neighborhood that is essential in $S_g.$ Thus, $u''_i$ is disjoint from $u_i$ and $u'_i$ and is therefore adjacent to both.

We conclude that $(u_i,u''_i,u'_i,u_{i+1})$ is a path in $\kfine(S_g).$

\p{Case 2: Bigons} The proof of this case is similar to that of Case 1. The only difference is that we obtain $u'_i$ by isotoping $u_i$ to remove an innermost bigon with $u_{i+1}$ such that the isotopy is done in a neighborhood of the bigon. We therefore have that $|u'_i\cap u_{i+1}|=k-1,$ so $u'_i$ is adjacent to $u_{i+1}$ in $\kfine(S_g).$ We obtain $u''_i$ analogously to Case 1. 

\p{Case 3: Essential intersections} We may now assume that $u_i$ and $u_{i+1}$ are in minimal position, as otherwise we could apply our work from cases 1 or 2. Thus, may construct a path of length two between $u_i$ and $u_{i+1}$ in $\mathcal{C}_k^\dagger(S_g)$ using conventional surgery techniques, such as those used for finding paths in the curve graph $\mathcal{C}(S)$ (see for example Masur--Minsky \cite[Lemma 2.1]{MM}).
\end{proof}

We are now ready to prove Theorem~\ref{maintheoremkfine}.

\begin{proof}[Proof of Theorem~\ref{maintheoremkfine}]
We proceed by induction. The fine curve graph is Gromov hyperbolic by Remark 3.2 and Theorem 3.8 of Bowden--Hensel--Webb \cite{Bowden_Hensel_Webb_2021}. By Proposition~\ref{prop:onefineqifine}, the fine curve graph is quasi-isometric to the fine 1-curve graph, so we have that the fine 1-curve graph is hyperbolic.

Suppose now that the $k$-curve graph is hyperbolic. By Proposition~\ref{prop:konefineqikfine}, the $(k+1)$-curve graph is quasi-isometric to the fine $k$-curve graph and is therefore hyperbolic, as desired.
\end{proof}

\section{Diameter of finitary curve graphs}\label{sec:finefinite}

In this section, we prove that the finitary curve graph has diameter 2, which is the content of Theorem~\ref{maintheoremfinitefine}. This implies that the finitary curve graph is quasi-isometric to a point (and therefore trivially hyperbolic). Along the way, we prove Proposition~\ref{prop:lotsofcurves} about the existence of a curve (or arc) that intersects a collection of other curves (or arcs) finitely many times each.

\p{Curves crossing an annulus} Let $\nu$ be an embedded annulus in $S_g$ bounded by two disjoint homotopic curves $v$ and $v'$. A curve $u$ \textit{crosses $\nu$ at $u_i$} if there exists a closed connected subset $u_i$ of $u$ such that $u\cap \nu$ is a single connected component and $|u_i\cap v| = |u_i\cap v'|=1.$ Suppose $u$ is the image of a map $f_u:[0,1]\to S_g$ that is an embedding on $(0,1)$ and $f(0)=f(1)$. If $u_i\cap \nu$ is the image of $[t_1,t_2]$, we call $f(t_1)$ the \textit{starting point} and $f(t_2)$ the \textit{endpoint}. Otherwise, if $u_i\cap \nu$ is the image of $[0,t_1]\cup [t_2,1]$, then $t_2$ is the starting point and $t_1$ is the endpoint.

Furthermore, $u$ \textit{crosses} $\nu$ $n$ \textit{times} if there exist $n$ (but not $n+1$) closed connected subsets $u_i$ of $u$, pairwise disjoint except potentially at their boundaries, such that $u$ crosses $\nu$ at $u_i$ for each $i.$ Examples of crossings can be seen in Figure~\ref{fig:loopsandcrossings}. We note that, a priori, $n$ can be infinity. The following lemma rules out this possibility. 

\begin{figure}[h]
\begin{center}
\includegraphics[width=4in]{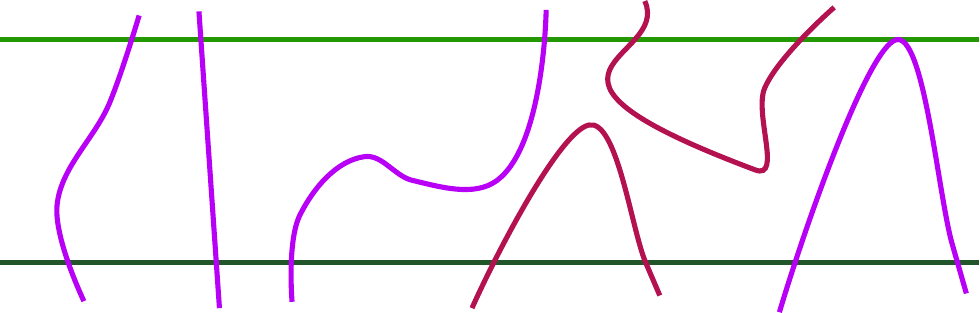}
\caption{We have an example of a curve $u$ that crosses the annulus bounded by green curves $v$ and $v'$ five times and forms two loops. The crossing strands are purple while the loops are burgundy.}\label{fig:loopsandcrossings}
\end{center}
\end{figure}

\begin{lemma}\label{lemma:finitecrossings}
Let $\nu$ be an embedded annulus in $S_g$ bounded by two homotopic curves $v$ and $v'$. Then, an embedded curve $u$ crosses $\nu$ finitely many times.
\end{lemma}

\begin{proof}
    Suppose $u$ crosses $\nu$ infinitely many times. By compactness, there must be a sequence $\{p_n\}$ of points in $u\cap v'$ that converges to some $p\in u\cap v'$. There exists a subsequence $\{p_{n_k}\}$ such that $p_{n_1},p_{n_2},\ldots,p_{n_k},\ldots,p$ appear sequentially along both $u$ and $v'.$ Fix a parametrization $f_u$ of $u$ such that $f_u^{-1}(p)=\{0,1\}$. We may again take a subsequence such that, up to reparametrization, each $p_i$ is the endpoint of some $u_i.$
    
    Abusing notation, we will call this subsequence $\{p_n\}.$ Let $q_i$ be the starting point of the $u_i$ for which $p_i$ is an ending point. Thus, we have that $q_1,p_1,q_2,p_2,\ldots$ appear sequentially along $u.$ Define $t_i = f_u^{-1}(p_i)$, $t=f_u^{-1}(p),$ and $s_i=f_u^{-1}(q_i).$ Up to taking a subsequence, we then have that $s_1<t_1<s_2<t_2<\cdots<s_i<t_i<\cdots$ is a bounded strictly increasing real sequence that has a subsequence that converges to $t.$ We then have that the entire sequence approaches $t$, so by continuity, $q_i=f_u(s_i)\to p.$ However, $\{q_i\}\subset v$, and there is a neighborhood of $v'$ disjoint from $v$, making this convergence a contradiction. 

    We conclude that $u$ cannot cross $v$ infinitely many times.
\end{proof}

A \textit{loop of $u$ in $\nu$} is a connected, closed subset of $u$ both of whose boundary points are in one of $v$ or $v'$ and whose interior is disjoint from both $v$ and $v'.$ In other words, a loop of $u$ is a portion of $u$ that bounds a bigon with $v$ or $v'.$ Examples of loops can be seen in Figure~\ref{fig:loopsandcrossings}.

We now prove the following lemma, a direct consequence of Lemma~\ref{lemma:finitecrossings}.

\begin{lemma}\label{lemma:finitefarloops}
    Let $\nu$ be an embedded annulus in $S_g$ bounded by two homotopic curves $v$ and $v'$ and let $u$ be a curve in $S_g.$ Then, finitely many loops of $u$ intersect any curve $v''$ contained in the interior of $\nu.$
\end{lemma}

\begin{proof}
    We first consider loops that intersect $v.$ Applying Lemma~\ref{lemma:finitecrossings} to the annulus bounded by $v$ and $v'',$ we obtain that finitely many loops cross this annulus. Therefore, finitely many loops that intersect $v$ also intersect $v''.$

    We repeat the above procedure with loops that intersect $v'$, thus considering all loops of $u$ in $\nu$ that intersect $v''.$
\end{proof}

We note that Lemmas~\ref{lemma:finitecrossings} and \ref{lemma:finitefarloops} also work if $\nu$ is a strip (a disk parameterized by $I\times I$ where $I\times \{0\},I\times \{1\}\subset \partial S$ and $v=\{0\}\times I$ and $v'=\{1\}\times I$) and $v$ and $v'$ are arcs that are homotopic (not rel boundary).

\p{Proof of Theorem~\ref{maintheoremfinitefine}} The key point of the proof of Theorem~\ref{maintheoremfinitefine} is contained in Lemma~\ref{lemma:crossingintonly}. Let $a$ be a curve in $S$ and $\alpha\subset a$ a closed, connected subarc with nonempty interior. Define a \textit{banana neighborhood} of $\alpha$ to be a disk $B$ such that the interior of $\alpha$ is contained in the interior of $B$,  $\alpha\cap \partial B = \partial \alpha,$ and $a\cap \overline{B}=\alpha.$
One way to construct such a B is to take a union of 1) $\bigcup U_x$ of simply connected neighborhoods $U_x$ for $x\in \alpha$ such that $a$ is locally flat in $U_x$ and $a \cap U_x \subset \alpha$ and 2) all disks bounded by $\overline{\bigcup U_x}.$ 

\begin{lemma}\label{lemma:crossingintonly}
    Let $S$ be any orientable surface. Let $\nu$ be an embedded annulus in $S$ (closed as a subset of $S$) and $\gamma_1,\ldots,\gamma_n$ a finite collection of curves and/or arcs in $S$ such that $\cup \gamma_i \cap \nu$ is a collection of arcs in $\nu$ disjoint from each other except potentially: 1) in $\partial \nu$ or 2) the $\gamma_i$ may coincide for the entirety of a connected component of $\nu\cap \bigcup \gamma_i$. Then, there is a curve $w\subset \nu$ such that $|w\cap \bigcup \gamma_i|<\infty.$ In particular, we can choose $w$ such that all intersections of $w$ with each $\gamma_i$ are crossing.

    The same can be done with $\nu$ a strip (disk) parameterized by $I\times I$ where $I\times \{0\},I\times \{1\}\subset \partial S$ and $v=\{0\}\times I$ and $v'=\{1\}\times I.$ 

    We note that if $\partial \nu$ (if $\nu$ is an annulus) or $\{0\}\times I$ (if $\nu$ is a strip) is essential in $S,$ then the constructed $w$ is essential in $S.$
\end{lemma}

\begin{proof}
    Let $v$ and $v'$ be the boundary curves of $\nu$ if $\nu$ is an annulus (or $v=\{0\}\times I$ and $v'=\{1\}\times I$ if $\nu$ is a strip). If either of $v$ or $v'$ satisfy the conditions on $w,$ we are done. Suppose neither is a suitable $w.$ For the remainder of the proof, we will use the language of curves an annuli, but an identical argument works for strips and arcs.
    
    For simplicity of notation, let $u=\bigcup \gamma_i.$ Since all $\gamma_i$ are disjoint or coincide in the interior of $\nu,$ we may continue to use the terms ``strand" and ``loop" of $u$ in $\nu,$ which will apply to the relevant $\gamma_i$.

    %Let $N_v$ and $N_{v'}$ be regular neighborhoods of $v$ and $v'$ with disjoint closures. Let $v''$ and $v'''$ be the boundary curves of $N_v$ and $N_{v'}$, respectively, that lie in $\nu$. Let $\nu'$ be the annulus bounded by $v''$ and $v'''$. Orient $v,v',v'',$ and $v'''$ such that they are all compatible with a predetermined orientation of $\nu$.
    Let $v''$ be a curve in the interior of $\nu$ and orient $v,\ v',$ and $v''$ in a compatible way with some pre-set orientation of $\nu.$

    We hope to construct a curve $w$ by following $v''.$ However, it may be the case that $|v''\cap u|=\infty.$ We must take this into consideration. First we will adjust $v''$ to ensure that $v''$ intersects all crossing strands of $u$ finitely many times. We will then adjust $v''$ to intersect each loop of $u$ in $\nu'$ finitely many times.

    \p{Adjusting $\mathbf{v''}$ around crossing strands}  Let $u_1,\ldots,u_n$ be the strands of $u$ that cross $\nu.$ We note that each $u_i$ is completely contained in some number of $\gamma_i.$ There are finitely many by Lemma~\ref{lemma:finitecrossings}. We will now amend $v''$ so that it intersects $\cup u_i$ at finitely many points. We do this by working with one strand at a time, beginning with $u_1.$

   \begin{figure}[h]
\begin{center}
\includegraphics[width=4in]{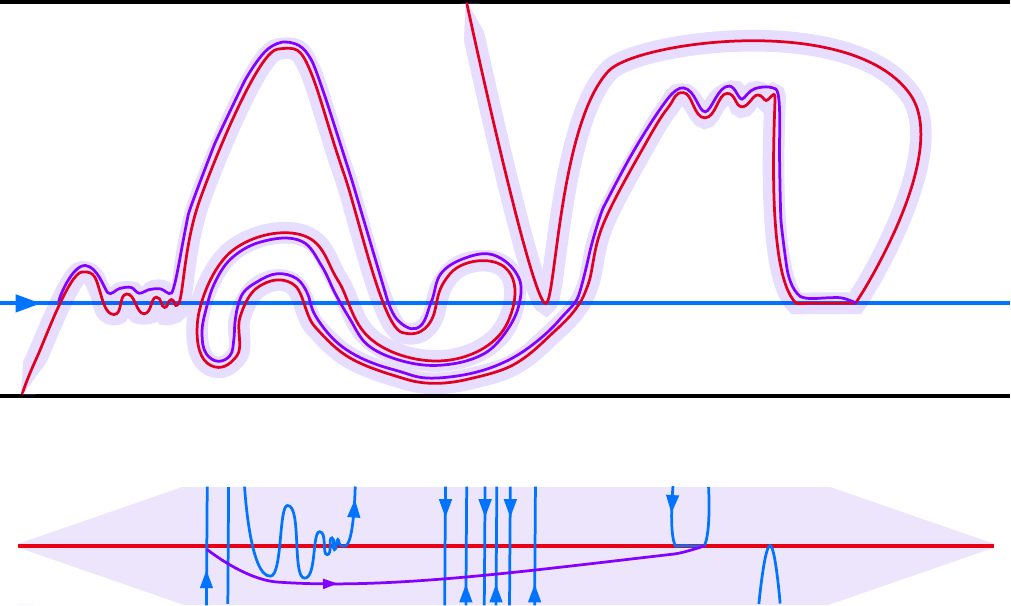}
\caption{Top: the horizontal lines are, top to bottom, $v',\ v'',$ and $v.$ Pictured as well are $u_1\subset u$ (red) and a banana neighborhood $B$ of $u_1$ that is disjoint from $v\cup v'$ except at the endpoints of $u_1.$ Bottom: the banana neighborhood $B$ of $u_1\subset u$. To ensure that we have a curve that intersects $u_1$ finitely many times, we surger $v''$ with the purple arcs, preserving the pictured orientation. In this image, we have not yet completed the final isotopy to remove the touching intersection.}\label{fig:crossingproof}
\end{center}
\end{figure}

   Let $x\in v''$ be such that $x$ is not contained in the closure of any bigon of $u_1$ with $v''.$ Thus, without loss of generality, we are in the case pictured in the top of Figure~\ref{fig:crossingproof} and $u_1$ is contained in a disk (pictured as a rectangle) in $\nu.$ We may parameterize $v''$ such that $v''(0)=x.$ Thus we have a total ordering on $v''\cap u_1$, where $v''(t_1)\leq v''(t_2)$ if and only if $t_1\leq t_2.$ Since $v''\cap u_2$ is compact, we have that there is a first and a last intersection of $v''$ with $u_1$; call these $x_F$ and $x_L$ (``first" and ``last").

    Let $B_1$ be a banana neighborhood of $u_1$ (disjoint from any $\gamma_i$ that are disjoint from $u_1$) and let $B_2\subset B_1$ be a banana neighborhood of the subset of $u_1$ between $x_F$ and $x_L$.   

    Beginning at $x,$ follow $v''$ in the direction of its orientation, and upon touching $x_F,$ begin to follow $\partial B_2.$ Upon the intersection with $x_L,$ again begin following $v''$ in the direction of its orientation until $x$ is reached once more. 

    The resulting curve or arc---which, abusing notation, we denote $v''$---is essential if $v$ was essential, as it is isotopic to $v$. Moreover, because $v''$ was adjusted only in a banana neighborhood of $u_1$ (and otherwise potentially completely removed), when this procedure is iterated for the other $u_i$, it does not create additional intersections with $u.$

    After the surgery, $|v''\cap u_1|=2$. However, as $u_1$ is a crossing strand in $\nu$, $v''$ must cross $u_1$ an odd number of times to be essential in $\nu.$ Thus, one of the intersections of $v''$ with $u_1$ must be touching. Isotope $u''$ in $B_1$ to get rid of this touching intersection. 

   Since there are finitely many strands $u_i$ of $u$ crossing $\nu,$ we have that $|v''\cap\bigcup u_i|<\infty$, as desired.

 \p{Adjusting $\mathbf{v''}$ around loops of $\mathbf{u}$ in $\mathbf{\nu}$} 
By Lemma~\ref{lemma:finitefarloops}, there are finitely many loops of $u$ in $\nu$ that intersect $v''.$ We may order them in any way and we will work with one of them at a time.

Every loop of $u$ in $\nu$ that intersects $\nu'$ has a banana neighborhood $B$ in $\nu$ disjoint from $v\cup v'\cup (u\setminus \text{loop})$ except potentially at the endpoints of the loop. Let $b\subset \partial B$ be the minimal subset of $\partial B$ that bounds the bigon with $v$ or $v'$ that contains the loop of $u.$ Orient $b$ such that if we flow a point outside the banana neighborhood along $v$ or $v'$, we could divert the flow to $b$.

Then $b$ has a first and last point of intersection with $v''$ in accordance with the orientation of $v''$. (To make the ideas of ``first" and ``last" precise, we may again parametrize $v''$ so that the image of $0$ is not contained in the bigon bounded by $b.$) We then surger $v''$ with $b$ and, abusing notation, call the resulting curve $v''$.

We do this finitely many times---at most once for each loop of $u$ in $\nu$ that intersects $v''$---and pick up no intersections with loops of $u$ in $\nu$ as a result. We note that depending on the ordering of the loops, the number of future loops may reduce by more than 1 after surgeries. Moreover, intersections with crossing strands of $u$ in $\nu$ are unaffected since all surgeries are done away from these strands.

\medskip\noindent After the above surgeries, we let $w=v''.$ We have that $w$ is disjoint from $v$ and intersects $u$ finitely many times, all intersections being crossing intersections, as desired.
\end{proof}

We are now ready to prove Theorem~\ref{maintheoremfinitefine}, which states that for any closed, orientable surface $S_g$ with $g\geq 2,$ we have that $\diam(\finfine(S_g))=2.$

\begin{proof}[Proof of Theorem~\ref{maintheoremfinitefine}]
    Let $a$ and $b$ be curves in $S_g.$ Let $\nu$ be an embedded annulus in $S_g$ with $a$ as a boundary component. Then, applying Lemma~\ref{lemma:crossingintonly}, we get a curve $c$ (isotopic and potentially equal to $a$) that intersects $b$ finitely many times.
\end{proof}

We further have that the following proposition directly follows from the above discussion.

\begin{proposition}\label{prop:lotsofcurves}
    Let $S$ be an orientable surface without punctures and $\Gamma = \bigcup \gamma_i$ with $\{\gamma_i\}_I$ a finite collection of curves and arcs in $S.$ Let $C\subset S$ be a closed subset of a Cantor set such that $\Gamma \setminus C$ is a collection of disjoint curves and arcs, and let $\Sigma= S\setminus C.$ 

    Then, there is a curve (or arc) $w\subset \Sigma$ with $w\not \in \{\gamma_i\}$ belonging to any isotopy class of curves or arcs in $\Sigma$ such that $|w\cap \gamma_i|<\infty$ for all $i\in I$. Moreover, each intersection of $w$ with each $\gamma_i$ is crossing.
    %Let $\Sigma$ be an orientable surface, potentially of infinite type, with $g\geq 2$. Let $\Gamma$ be a finite collection of disjoint arcs and curves in $\Sigma.$ Then there is a curve $w\not\in \Gamma$ (belonging to any isotopy class of simple closed curves or properly embedded arcs) that intersects each curve and arc in $\Gamma$ in at most finitely many points. Moreover, each intersection of $w$ with $\Gamma$ is crossing.
\end{proposition}

We now have the following corollary on account of the finite diameter of $\finfine(S_g).$

\begin{corollary}
Let $S_g$ be a closed, orientable surface with $g\geq 2.$ Then, $\finfine(S_g)$ is quasi-isometric to a point (and therefore trivially hyperbolic).
\end{corollary}

We may also recover the following well-known result.

\begin{corollary}
    For a closed, orientable surface $S_g$ with $g\geq 2,$ the graph $\varinjlim_k \mathcal{C}_k(S_g)$ is complete. That is, any two isotopy classes of curves admit representatives that intersect finitely many times.
\end{corollary}

\section{Contractibility of $\mathbf{\finfine(S_{g})}$}\label{sec:contractability}

In this section, we will show that the flag complex induced by $\finfine(S_g)$ is contractible using Whitehead's theorem by showing that every sphere is contractible. We will take this complex to have as its $k$-cells sets of $k+1$ vertices representing curves that pairwise intersect finitely many times.

In particular, let $\gamma_1,\ldots,\gamma_n$ be the vertices of the image of a sphere in $\finfine(S_g)$. We will show that there exists a vertex $\gamma$ adjacent each $\gamma_i$, so the sphere is in the link of a vertex. 

We define an \textit{infinite-type surface} to be a surface whose fundamental group cannot be finitely generated. A surface of finite type is any surface that is not of infinite type.

We will now define the set of ``messy" intersections between a finite set of curves and arcs. Given curves or arcs $\gamma_1,\ldots,\gamma_n$ in a surface $\surf,$ we define $\mathcal{E}(\gamma_1,\ldots,\gamma_k)\subset \surf$ in the following way. (We will abuse notation by treating the $\gamma_i$ both as injective maps from $S^1 = I/\{0\}\sim\{1\}$ or $I$ into $\Sigma$ and as the images of such maps; the use should be clear from context.)
    
Let $E=\{p\in S | p\in \gamma_i\cap \gamma_j \text{ for some }i,j\};$ these are all the points of intersection between the arcs and curves. However, we only want the potentially messy intersections, not the intersections where two curves or arcs coincide for an open interval. Let $P_i=\partial \{\gamma_i^{-1}(p) | p\in E\}\subset I;$ we note that $P_i$ is empty if $\gamma_i$ is disjoint from all other $\gamma_j$. Moreover, $P_i$ is totally disconnected since it is the boundary of a subset of the interval. To get these points back on the surface, let $\E=\bigcup_i \{\gamma_i(P_i)\}.$ We then have the following result about any such $\E.$

\begin{lemma}\label{lemma:finitelymanyintersections}
    Let $\Gamma=\{\gamma_1,\ldots,\gamma_n\}$ be a collection simple closed curves and/or arcs in $\Sigma,$ an orientable and potentially infinite-type surface. Then there is a curve $\gamma$ such that $|\gamma\cap \gamma_i|<\infty$ for all $i.$ Moreover, all intersections between $\gamma$ and $\gamma_i$ are crossing.
\end{lemma}

\begin{proof}
    The goal is to construct a curve disjoint from the ``messy" intersections between the existing curves and arcs. We claim that $\E$ is compact (and therefore closed) and totally disconnected.

    First, $\E$ is a finite union of compact sets and is therefore itself compact. Second, suppose $C$ is a connected component of $\E$ with more than one point. Then at least one of $\gamma_i(P_i)$ is not nowhere dense in $C.$ Let $C'\subset C$ be a connected open subset of $C$ (via the subspace topology) in which $\gamma_i(P_i)$ is dense. Note that $C'$ must have more than one point. Since $\gamma_i(P_i)$ is closed, $\gamma_i(P_i)\cap C'=C'.$ This contradicts the fact that $\gamma_i(P_i)$ is totally disconnected, concluding the proof of the claim. 
    
    Since $\E$ is compact and totally disconnected, we have that $\Sigma' = \Sigma\setminus\E$ is a connected surface (potentially of infinite type and potentially with noncompact boundary). In particular, the universal cover of $\Sigma\setminus \E$ is the hyperbolic plane, which is simply connected, so $\Sigma'$ must be connected. We further have that the images of  $\Gamma\setminus\E$ comprise a collection of \textit{disjoint} arcs in $\Sigma'$, where at most a finite number of $\gamma_i$ coincide at any arc. Abusing notation, we call this collection of arcs $\Gamma$ as well.

    Let $v$ be a curve in $\Sigma'.$ By Lemma \ref{prop:lotsofcurves}, we can perturb $v$ so that it intersects $\Gamma$ at finitely many points with each intersection being crossing. We call the perturbed version $\gamma.$
\end{proof}

\begin{proof}[Proof of Theorem~\ref{maintheoremcontractibility}]
    Let $\gamma_1,\ldots,\gamma_n$ be the vertices of the image of a sphere in $\finfine(S_g)$. By Lemma~\ref{lemma:finitelymanyintersections}, there is a curve $\gamma$ adjacent to all the $\gamma_i$s. Therefore, the sphere is in the link of $\gamma.$ Since $\finfine(S_g)$ is a flag complex, the link of $\gamma$ is contractible. Applying Whitehead's theorem, we have that $\finfine(S_g)$ is contractible.
\end{proof}

In fact, if $\Sigma$ from Lemma~\ref{lemma:finitelymanyintersections} is a compact finite-type surface $S$ with a subset of a Cantor set $P\subset S$ removed, then we have the following corollary.

\begin{corollary}\label{cor:choosecurve}
    Let $S$ be a surface of finite type and $P\subset S$ be a closed subset of a Cantor set. Let $\Sigma = S\setminus P.$ Let $\Gamma=\{\gamma_1,\ldots,\gamma_n\}$ be a finite collection of essential simple closed curves and/or arcs in $\Sigma,$ where all arcs must have their endpoints in $\partial \Sigma.$ Then there is a curve or arc $\gamma$ such that
    \begin{enumerate}
        \item $|\gamma\cap \gamma_i|<\infty$ for all $i$,
        \item all intersections between $\gamma$ and $\gamma_i$ for all $i$ are crossing, and
        \item $\gamma$ belongs to any pre-selected isotopy class of essential simple closed curves and/or arcs in $S$ and is arbitrarily close to any pre-selected representative of said class. (If $\gamma$ is an arc, we add the restriction that all endpoints must be contained in $\partial \Sigma.$)
    \end{enumerate}
\end{corollary}

\begin{proof}

    Let $\alpha$ be a representative of the desired isotopy class in $S$; we note that $\alpha \cap P$ might not be empty. Consider an annular or strip neighborhood $A$ of $\alpha.$ Then, $A\setminus P$ is a surface, potentially of infinite type, and is therefore path connected. 

    If $\alpha$ is an arc, there is a path in $A\setminus P$ connecting the boundary component(s) $\alpha$ intersects. 

    If $\alpha$ is a curve, cut $A=S^1 \times I$ into a strip via $L=\{0\}\times I.$ Let $x\in L\setminus P.$ Since the resulting strip (once cut open) punctured at $P$ is a surface of infinite type and there are 2 images of $x$ under the cutting operation, there is an arc that joins the two images of $x$ that stays inside the strip. Back on the surface, this path glues up to a curve isotopic to $\alpha$ and disjoint from $P,$ as desired.
    %
    %Since $P$ is a subset of a Cantor set, it is contained in an embedded disk. (This can be seen by using the classification of infinite-type surfaces and taking $P$ to a Cantor subset of any clearly visible embedded disk \cite{richards}.) Any curve in $S$ can be isotoped to avoid said embedded disk.
\end{proof}

\section{Subgraphs of the fine curve graph and the finitary curve graph}\label{sec:subgraphs}

In this section, we prove results about admissibility of subgraphs of $\mathcal{C}_k^\dagger(S_{g,b})$ and $\finfine(S_{g,b}).$ In Section~\ref{subsec:admissible}, we focus on subgraphs that are admissible, addressing Theorems~\ref{maintheorem:finite.in.fine} and \ref{maintheorem.countable.subgraphs.in.finitary}. In Section~\ref{subsec:forbidden}, we focus on subgraphs that are inadmissible in fine curve graphs, proving Theorem~\ref{maintheorem:forbiddensubgraphs}.

\subsection{Admissible subgraphs of fine curve graphs}\label{subsec:admissible}

In this section, we prove Theorems~\ref{maintheorem:finite.in.fine} and \ref{maintheorem.countable.subgraphs.in.finitary}. Let $G$ be a finite graph. First, we prove that for high enough genus, $G$ is isomorphic to an induced subgraph of the fine curve graph of a sufficiently complex surface.

\begin{proof}[Proof of Theorem \ref{maintheorem:finite.in.fine}]
    Let $\gamma_0,\ldots,\gamma_{n-1}$ be $n$ parallel disjoint curves in the torus, drawn in cyclical order, as in the left of Figure~\ref{fig:finitegraphinfine}. For each $\gamma_i,$ consider $n-2$ disjoint closed disks that intersect only that $\gamma_i$ and do so at one point on the boundary. Now, remove the interior of each disk and glue in an annulus between any two curves $\gamma_i,\gamma_j$ if $|i-j|\geq 2$ and  $\{i,j\}\neq \{0,n-1\},$ as in the center of Figure~\ref{fig:finitegraphinfine}. We call these annuli \emph{handles} and label them $A_{i,j}$ if $i<j.$

\begin{figure}[h]
\begin{center}
\includegraphics[width=1.75in]{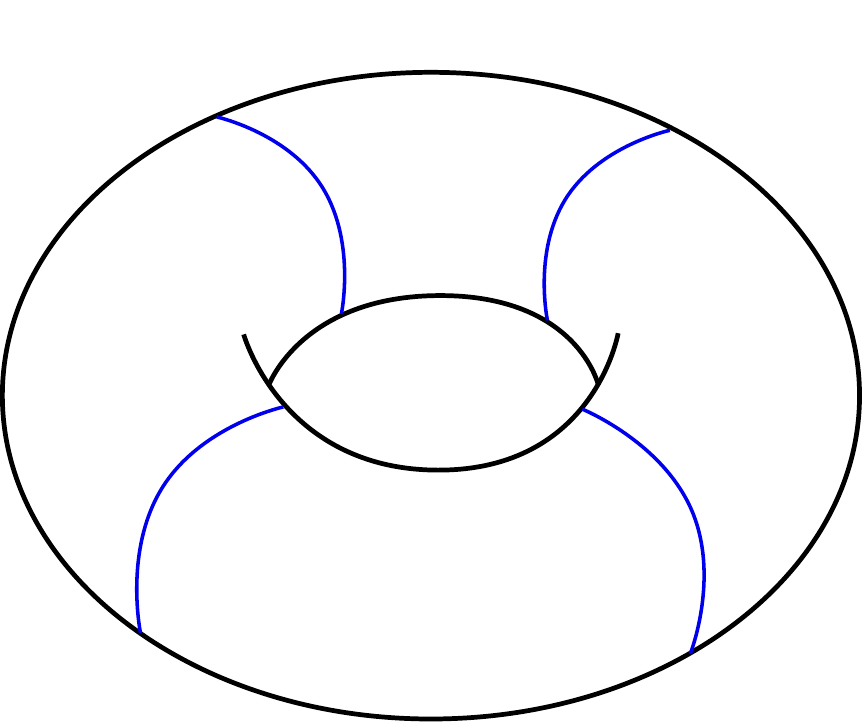}
\includegraphics[width=1.75in]{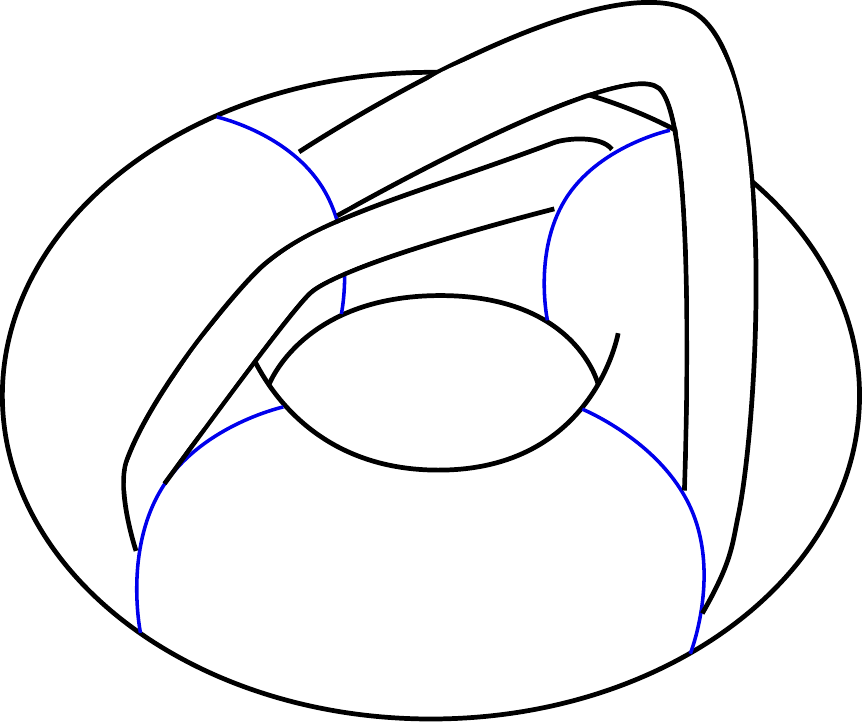}
\includegraphics[width=1.75in]{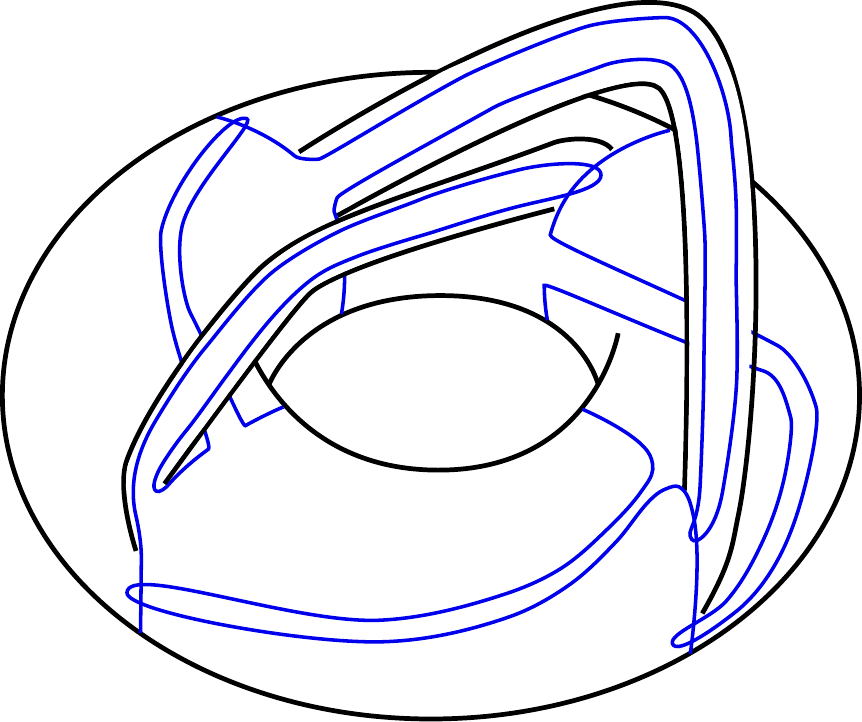}
\caption{Left: Four curves on a torus. Center: Handles (annuli) attached connecting curves that are not directly next to each other. Right: An example of four curves that induce a subgraph of $\fine(S_{1,0})$ isomorphic to the graph on four vertices with no edges.}\label{fig:finitegraphinfine}
\end{center}
\end{figure}

    We now work with adjacency in $G.$ Let $v_0,\ldots,v_{n-1}$ be the vertices of $G.$ If $v_i$ is \textit{not} adjacent to $v_j$ and $i<j,$ we either 1) isotope $\gamma_i$ in a small neighborhood of $A_{i,j}$ disjoint from all other $\gamma_k$ and $A_{i',j'}$ to intersect $\gamma_j$ or 2) isotope $\gamma_i$ to intersect $\gamma_j=\gamma_{i+1}$ away from all the attached handles. An example of a set of curves representing a graph on vertices with no edges can be seen on the right of Figure~\ref{fig:finitegraphinfine}.

    We therefore have perturbed curves $\gamma'_0,\ldots,\gamma'_{n-1}$ such that the subgraph of $\finfine(S_g)$ induced by said curves is isomorphic to $G.$

    To calculate the genus of the constructed surface, notice that attaching one handle increases genus by 1. We are attaching ${n \choose 2}-n$ handles, leading to a total genus of ${n\choose 2}-n+1.$
\end{proof}

We note that the bound on genus in the proof of Theorem~\ref{maintheorem:finite.in.fine} is not sharp; in particular, all graphs on 5 vertices (as shown in Appendix~\ref{appendix:smallgraphsintorus} and all but one graphs on 6 vertices are admissible as induced subgraphs of $\fine(S_1).$

Moreover, the naive $O(g^2)$ bound of Theorem~\ref{maintheorem:finite.in.fine}, though found independently, actually matches known asymptotic bounds for subgraphs of curve graphs \cite{bg}.

We now work with the fine $k$-curve graphs ($k\geq 2$) and finitary curve graph, wherein we prove that any countable graph $G$ is isomorphic to an induced subgraph of $\finfine(\surf)$  and $\mathcal{C}_k^\dagger(\surf)$ for $k\geq 2$. 

\begin{proof}[Proof of Theorem~\ref{maintheorem.countable.subgraphs.in.finitary}]

    (The proof below can be more easily understood by looking at Figure~\ref{fig:anyinfgraph}.) 
    Enumerate the vertices in $G$ so that the vertex set is $v_1,v_2,\ldots.$ Let $A\subset \surf$ be an embedded annulus whose boundary components are essential and non-peripheral (not isotopic to a component of $\partial \surf$). Parameterize $A$ by $A=I\times S^1 = I \times (I/\{0\}\sim \{1\}).$ Let $\{w_n\}_{n\in\mathbb{N}}$ be the set of curves with $w_n=\frac{1}{n+1}\times S^1.$ All of the $w_n$ are disjoint.

    Our goal now is to make all of the $w_n$ intersect such that 1) $|w_i\cap w_j|=2$ for all $i,j$ and 2) each point of $w_i\cap w_j$ has a neighborhood disjoint from all other $w_k.$ In particular, to each $w_n$, $n\geq 2,$ associate a disk $A_n=(\frac{2n+3}{2(n+1)(n+2)},1)\times (\frac{2n+3}{2(n+1)(n+2)},\frac{2n+1}{2n(n+1)})\subset A$. All of these disks are disjoint. Isotope each $w_n$ in $A_n$ to intersect each $w_i$ with $i<n$ in exactly two points, as in Figure~\ref{fig:anyinfgraph}. Thus we have a new collection $\{w'_n\}$ of curves that are still all adjacent to each other.

    \begin{figure}[h]
\begin{center}
\includegraphics[width=\textwidth]{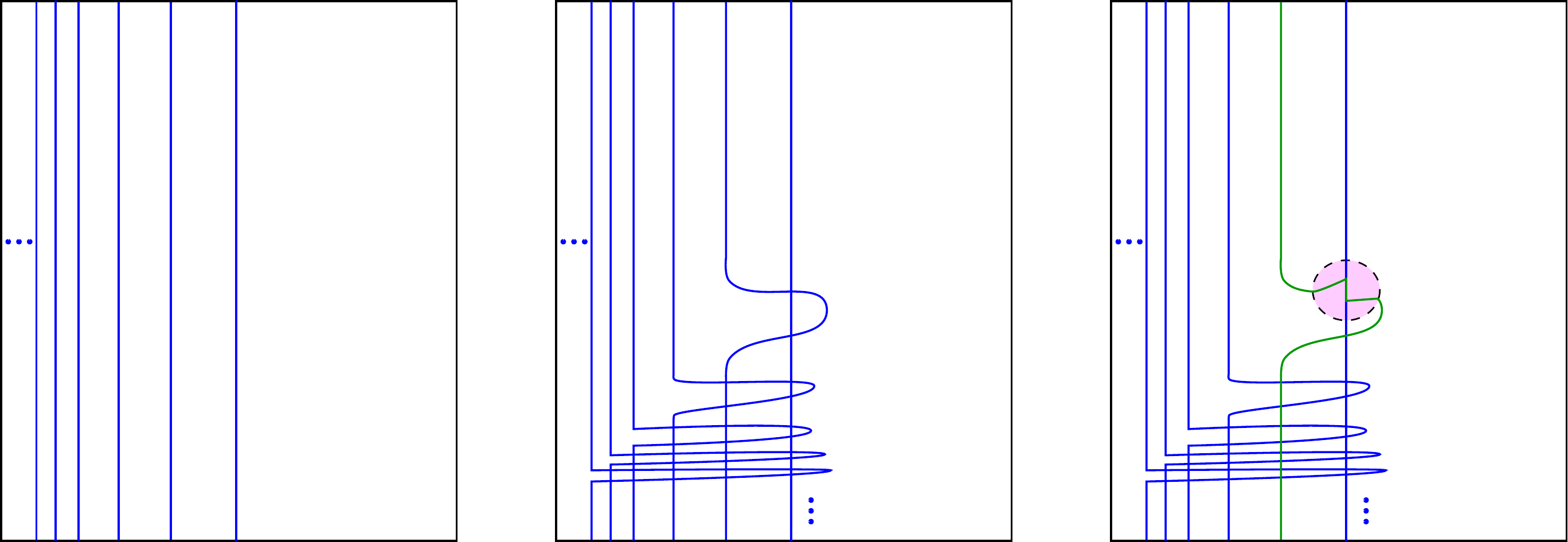}
\caption{Left: A schematic of $A=I\times (I/\{0\}\sim \{1\})$ with curves $\{w_n\}$. Center: A schematic of $A$ with each $w_j$ isotoped to intersect each $w_i$ with $i<j$ in exactly two points, creating the collection $\{w'_n\}$. Right: A perturbation of $w'_2$ in $P_{1,2}$ so that it intersects $w'_1$ infinitely many times. This is done in the case that $v_1$ is not adjacent to $v_2$ in $G.$ Thus, the images $u_1$ and $u_2$ of $w'_1$ and $w'_2$ after perturbation are not adjacent in $\finfine(\surf).$}\label{fig:anyinfgraph}
\end{center}
\end{figure}

    Let $\{P_{i,j}\}$ be a family of disjoint open sets with $P_{i,j}\subset A_i$ an open neighborhood of a point of $w'_i\cap w'_j$ that does not intersect any other $w'_k.$ If $v_i$ is not adjacent to $v_j$ for $i<j,$ isotope $w'_j$ in $P_{i,j}$ to intersect $w'_i$ infinitely many times. Once $w'_j$ has been isotoped in this way for all $i<j$ with $v_i$ not adjacent to $v_j$, call it $u_j.$ We note that all curves need to be isotoped only finitely many times, and due to the restrictions on the $P_{i,j}s,$ the isotopies do not create extra intersections with unrelated curves. 

    Define $f:G\to \finfine(\surf)$ by $f(v_i)=u_i.$ We therefore have that $u_j=f(v_j)$ is adjacent to $u_i=f(v_i)$ if and only if $v_i$ is adjacent to $v_j,$ as desired.

\end{proof}

We now turn our attention to the Erd\H{o}s-R\'enyi graph. The Erd\H{o}s-R\'enyi graph, which we denote $ERG,$ is the unique random on countably infinitely many vertices. This fact is proven using the following property of $ERG$.

\begin{center}
    ($*$) Given finitely many distinct vertices $u_1\ldots,u_m,v_1\ldots,v_n,$ there is a vertex $z$ adjacent to each $u_i$ and not adjacent to each $v_j.$
\end{center}

We show that $\finfine(\surf)$ does \textit{not} have Property ($*$) but nevertheless has $ERG$ as a subgraph. We find a collection of curves $v_1,\ldots,v_n,u$ for which there is no $z$ adjacent to $v_1,\ldots,v_n$ but not adjacent to $u$ in Figure~\ref{fig:counterexample}.

\begin{figure}[h]
\begin{center}
\includegraphics[width=2in]{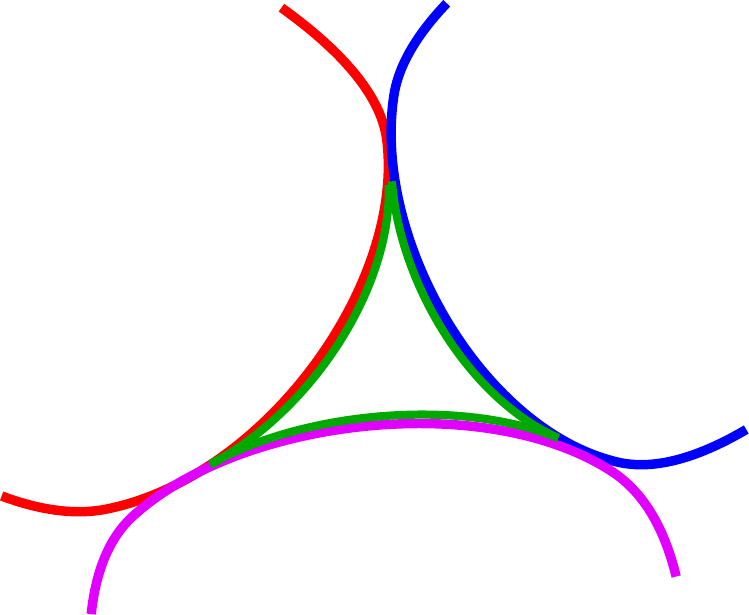}
\caption{A schematic of a curve contained in the union of other curves.}\label{fig:counterexample}
\end{center}
\end{figure}

What makes the above a counterexample is that $u\subset \cup v_i.$ If we construct a subgraph such that no curve is a subset of the union of the other curves, we will have a chance at having a subgraph of $\finfine(\surf)$ isomorphic to $ERG.$ %We do this in the proof of Theorem~\ref{maintheorem.countable.subgraphs.in.finitary}. Corollary~\ref{cor.erg.in.finitary} follows from the construction.

Now that we have fully classified countable graph admissibility in $\kfine(\surf)$ for $k\geq 2$ and $\finfine(\surf),$ we turn our attention to $\fine(\surf).$ We introduce a lemma related to recursively building admissible subgraphs of $\fine(\surf).$

\begin{lemma}[Inductive admissible subgraphs]\label{combo:lemma:subgraphadmissibility}
    Let $G$ be a finite graph that is realized as an induced subgraph of $\fine(\surf).$ Then the following graphs are also admissible.
    \begin{enumerate}
        \item (Disjoint union) $G\sqcup v$, the disjoint union of $G$ with an isolated vertex $v$
        \item (Single vertex attachment) $G\cup v,$ where the degree of $v$ is 1
        \item (Copycat vertex) $G \cup v,$ where $v$ is a \textit{copycat} vertex; that is, there exists a vertex $w\in G$ such that the neighborhood of $w$ equals the neighborhood of $v$ and $w$ is not adjacent to $v$
        \item (Blowup to a clique) $G \cup K_n$, where $K_n$ is a blowup of some vertex $w$ in the sense that $w$ is replaced with $K_n$ and each $v\in K_n$ has the same adjacencies as $w$ 
        \item (Cone vertex to a clique) $G\cup K_n,$ where $K_n$ is a clique on $n$ vertices and there exists a single vertex $w\in G$ such that every vertex in the $K_n$ is adjacent to only that vertex in $G$ (so $K_n\cup \{w\}$ is a $K_{n+1}$)
        %\item (Wedge of graphs) Let $G_1$ and $G_2$ be admissible in $\fine(S_g),$ and $\gamma_1\in V(G_1)$ and $\gamma_2\in V(G_2).$ Suppose there is a realization of $G_1$ and $G_2$ such that $\gamma_1$ and $\gamma_2$ are both non-separating. Then $G=G_1 \cup G_2$ by identifying $\gamma_1$ and $\gamma_2$ is admissible in $\fine(\surf)$.
    \end{enumerate}
    Moreover, we have a strengthening of (1) and (2) above that follow directly from (1) and (2):
    \begin{enumerate}
        \item[1'.] $G \sqcup_i v_i,$ where $\{v_i\}$ is a finite collection of isolated vertices.
        \item[2'.] $G \cup T,$ where $T$ is a tree and $T$ is attached to $G$ by 1 edge.
    \end{enumerate}
\end{lemma}

\begin{proof}
    Let $\phi:G\to \fine(\surf)$ be an injective homomorphism such that the induced subgraph of $\phi(V(G))$ is isomorphic to $G.$ We will conflate $\phi(G)$ with $G$ in what follows.
    We prove each part in turn.
    
    \pit{1. Disjoint union} We may take any curve $\gamma$ on $\surf$ and isotope it to intersect all of the curves that comprise the vertices of $\phi(G)$ to create the curve $\gamma'.$ We may then extend $\phi$ by $\phi(v)=\gamma'.$ (1') follows by recursively adding isolated vertices.

    \pit{2. Single vertex single attachment} Suppose $v$ is adjacent to $u\in G.$ Then, cut the surface $\surf$ along $\phi(u)$. We then apply (1) to $S_g\setminus \phi(u):$ take any curve $\gamma$ in $S_g\setminus \phi(u)$ and isotope it to intersect all other curves in $\phi(V(G -\{u\})),$ resulting in the curve $\gamma'.$ We may then extend $\phi$ by $\phi(v)=\gamma'.$ (2') follows by recursively adding vertices from the tree, beginning at a chosen root and proceeding by vertex depth in the tree.

    \pit{3. Copycat vertex} Let $w\in V(G).$ Let $V_w^c = \{v\in V(G) | v \text{ is not adjacent to }w\}$. We will construct a copycat vertex $w'$ of $w.$ Let $\{p_i\}_{I=|V_w^c|}$ be a collection of one point in every nonempty $\phi(w)\cap \gamma_i$ for each $\gamma_i\in \phi(V_w^c).$ Moreover, let $N_i$ be a neighborhood of each $p_i$ such that $\phi(w)\not \subseteq \cup N_i$. This is possible since $|V(G)|<\infty$ (and since all intersections are compact sets). Let $N$ be a neighborhood of a point in $\phi(w)\setminus\cup N_i$. We ask that $N$ is disjoint from any curve in $\phi(V(G))$ from which $w$ is disjoint. 

     Isotope $w$ in $N$ to create $w'=\phi(v),$ which has have the same adjacency properties as $w$ by construction. Be sure that $w'$ still intersects $w$ after the isotopy.

     \pit{4. Blowup to a clique} By Lemma~\ref{combo:lemma:niceadmissiblesubgraphs}, we may assume that all intersections between curves in $\phi(V(G))$ are crossing. Let $v\in G.$ Then, there is an annular neighborhood $A$ of $\phi(v)$ such that if $u$ is adjacent to $v,$ then $\phi(u)\cap A = \emptyset$ and if $u$ is not adjacent to $v,$ then there is a crossing strand of $\phi(u)$ in $A.$ In particular, any curve parallel to $\phi(v)$ contained in $A$ intersects $\phi(u).$ Insert $n$ parallel copies of $\phi(v)$ in $A$ and set them as the image of the $K_n$ that $v$ is replaced with.
     
    \pit{5. Cone vertex to clique} By (2), for one vertex $v$, $G\cup \{v\}$ is admissible in $\fine(\surf).$ By (4), we can blow $v$ up to a clique $K_n.$

\end{proof}

%Note that (6) above actually implies (2), (4) and (5).

\subsection{Inadmissible subgraphs of fine curve graphs}\label{subsec:forbidden}

In this subsection, we prove Theorem~\ref{maintheorem:forbiddensubgraphs} about the inadmissibility of finite subgraphs of $\surf$.

A naive guess may be to consider finite graphs that are inadmissible in curve graphs. Perhaps the most simple graph inadmissible in $\surf$ is a complete graph on $3g+b-2$ vertices since a maximal collection of disjoint curves in $\surf$ consists of $3g+b-3$ curves (corresponding to a pants decomposition). However, a complete graph on uncountably many vertices can be realized in an annulus (allowing for boundary parallel curves). This further implies that clique number (the size of the largest complete subgraph) and chromatic number are not barriers for induced subgraphs since both are infinite for fine curve graphs.

The next guess may be the \emph{half-graphs} of Bering--Conant--Gaster: bipartite graphs whose vertices are partitioned into 2 sets, $\{v_1,\ldots,v_N\}$ and $\{w_1,\ldots,w_N\}$ while $v_i$ is adjacent to $w_j$ if and only if $j\geq i$. Bering--Conant--Gaster prove that for a given surface, there exists a half-graph of sufficiently high complexity that is inadmissible in $\mathcal{C}(\surf)$ \cite{bcg}. However, an arbitrarily large half-graph can be realized in an annulus, as shown in Figure~\ref{fig:combo:halfgraph}. This proves that the complexity measure of Bering--Conant--Gaster that precludes finite graphs from being admissible in the curve graphs of surfaces does not apply in the case of fine curve graphs. (Moreover, this implies that, unlike curve graphs, fine curve graphs are not $k$-edge stable for any $k$, impacting their stability in the model theoretic sense; see Bering--Conant Gaster \cite{bcg} and Disarlo--Koberda--de la Nuez Gonzales \cite{dkn}.)

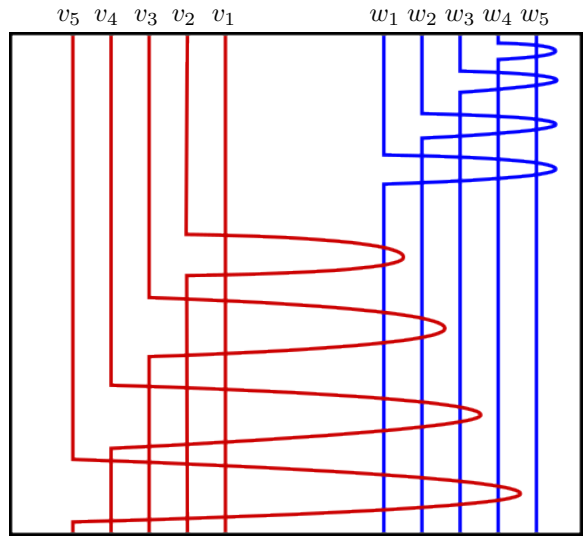
\begin{figure}[h]
\begin{center}
\begin{tikzpicture}
    \node[anchor = south west, inner sep = 0] at (0,0) {\includegraphics[width=03in]{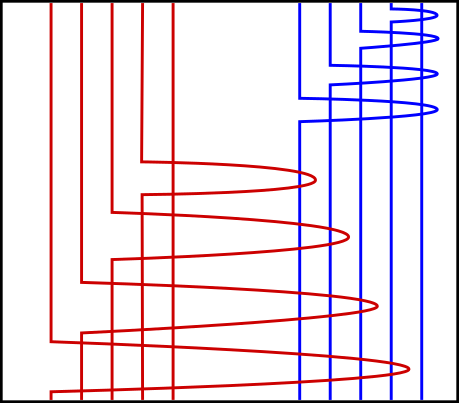}};
    %\draw[help lines] (0,0) grid (8,7);
    \node at (0.8,6.9){$v_5$};
    \node at (1.3,6.9){$v_4$};
    \node at (1.83,6.9){$v_3$};
    \node at (2.33,6.9){$v_2$};
    \node at (2.86,6.9){$v_1$};

    \node at (5,6.9){$w_1$};
    \node at (5.5,6.9){$w_2$};
    \node at (6,6.9){$w_3$};
    \node at (6.5,6.9){$w_4$};
    \node at (7,6.9){$w_5$};

    \end{tikzpicture}
\caption{A realization of a half-graph on 10 vertices in an annulus (the top and bottom edges of the rectangle are identified). A similar realization exists for arbitrarily large half-graphs.}\label{fig:combo:halfgraph}
\end{center}
\end{figure}

We note, however, that the construction of Bering--Gaster for graphs inadmissible in curve graphs of surfaces with finite genus does hold in the fine curve graph case \cite{bg}.

We will approach the proof of Theorem~\ref{maintheorem:forbiddensubgraphs} in parts. First, we will find inadmissible subgraphs for an annulus, then a torus (a singular case), then a pair of pants. We will then use the case of the pair of pants to build larger inadmissible subgraphs. To define fine curve graphs in these cases: for annuli and pairs of pants, we allow vertices to be boundary-parallel properly embedded closed curves, while for tori, annuli, pairs of pants, and four-holed spheres, we take edges to connect precisely those vertices that are disjoint.

To begin, we prove that all admissible subgraphs can be realized via nicely interacting curves.

\begin{lemma}[Nice admissible subgraphs]\label{combo:lemma:niceadmissiblesubgraphs}
    If a finite graph $G$ is an admissible subgraph of $\fine(S_g),$ then there is an injective graph homomorphism $\phi:G\to \fine(S_g)$ with $G\cong \phi(G)$ such that all curves in $\phi(V(G))$ pairwise intersect finitely many times and with only crossing intersections.
\end{lemma}

\begin{proof}
    We will conflate $G$ with its image in $\fine(S_g).$ Suppose the vertices of $G$ are $\gamma_1,\ldots,\gamma_n.$ We will ensure that there is another induced subgraph of $\fine(S_g)$ isomorphic to $G$ such that the vertices are all curves that pairwise have only crossing intersections by adjusting the image of a preexisting injective graph homomorphism.
   
    We begin with $\gamma_1$. There is an annular neighborhood $A$ of $\gamma_1$ that is disjoint from all $\gamma_i$ that $\gamma_1$ is disjoint from. By Corollary~\ref{cor:choosecurve}, there exists a curve $\gamma'_1$ in $A$ (and thus isotopic to $\gamma_1$) such that $\gamma'_1$ intersects all other $\gamma_i$ finitely many times and only at crossing intersections. We note that $\gamma'_1$ is adjacent to all curves $\gamma_1$ is adjacent to, but may in fact be adjacent to too many curves.  We must therefore reintroduce intersections between $\gamma'_1$ and other curves. Without loss of generality, suppose we are reintroducing intersections with $\gamma_2.$

    First, we know $\gamma_2\cap A\neq \emptyset$ since $\gamma_2\cap \gamma_1 \neq \emptyset.$ Let $\mathcal{E}=\mathcal{E}(\gamma'_1,\gamma_2,\ldots,\gamma_n).$ Consider $A'=A\setminus \mathcal{E},$ so the $\gamma_i$ become arcs in $A'.$ Connect $\gamma'_1$ to $\gamma_2$ in $A'$ by a path $\alpha$ disjoint from $\gamma'_1$ and $\gamma_2$ except at its endpoints. This is possible because any path from $\gamma'_1$ to $\gamma_2$ has a first intersection with $\gamma_2$ and a preceding intersection with $\gamma'_1$; we take the subpath between these intersections to be $\alpha.$ By cutting $A'$ along $\gamma'_1$ and $\gamma_2,$ we can apply Corollary~\ref{cor:choosecurve} to find 2 disjoint, arbitrarily close parallel paths $\alpha_1$ and $\alpha_2$ connecting $\gamma'_1$ and $\gamma_2$ that intersect all other $\gamma_i$ finitely many times each (at crossing intersections). We then connect these two paths by a third path $\alpha_3$ near $\gamma_2$ (and disjoint from all other $\gamma_i$) such that the combined path $\alpha_1,\alpha_3,\alpha_2$ crosses $\gamma_2$ exactly twice. Finally, we surger $\gamma'_1$ with this path to create a curve $\gamma''_1$ that (1) remains disjoint from all curves from which $\gamma_1$ is disjoint, (2) intersects all $\gamma_i$ at most finitely many times, and (3) intersects $\gamma_2.$

    We repeat the above surgeries to ensure that $\gamma_1$ has the correct intersection patterns before moving on to $\gamma_2$ and performing the same algorithm. (It is possible that at each step of the process, we may redo work that was done in a previous step.)

    Since each step is done finitely many times, it follows that all curves intersect each other finitely many times and no new non-crossing intersections are introduced. Moreover, since one curve is being isotoped at a time, all intersections are preserved. Therefore, the resulting isotoped curves comprise an induced subgraph of $\finfine(\surf)$ isomorphic to $G.$
\end{proof}

\begin{lemma}[Inadmissible graphs, the annular case]\label{prop:combo:inadmissibleannulus}
    The following two graphs are inadmissible as subgraphs of the fine curve graph of an annulus.
    \begin{enumerate}
        \item $G_1=$ a $2n+1$-cycle for $n\geq 2$, as pictured in the top left of Figure~\ref{fig:combo:inadmissiblegraph1} without the central (unlabeled) vertex.
        \item $G_2=$ the graph pictured in Figure~\ref{fig:combo:inadmissiblegraph2} without the central (unlabeled) vertex.
    \end{enumerate}
\end{lemma}

\begin{figure}[h]
\begin{center}
\begin{tikzpicture}
    \node[anchor = south west, inner sep = 0] at (0,0) {\includegraphics[width=5in]{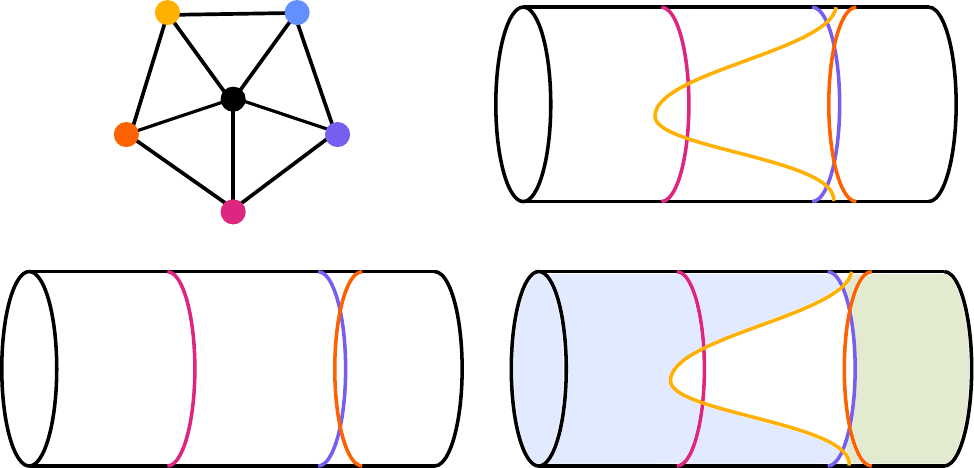}};
    %\draw[help lines] (0,0) grid (13,7);
    \node at (4.8,4.4) {$v_1$};
    \node at (3.65,3.3) {$v_2$};
    \node at (1.18,4.4) {$v_3$};
    \node at (2.2,6.27) {$v_4$};
    \node at (3.9,6.27) {$v_5$};

    %lower left
    \node at (2.2,1.2) {$v_2$};
    \node at (4.1,1.2) {$v_3$};
    \node at (4.8, 1.2) {$v_1$};

    %upper right
    \node at (10,5) {$v_4$};
    \end{tikzpicture}
\caption[A graph on 6 vertices inadmissible as a subgraph of $\fine(S_1).$]{We conflate $v_i$ with its image under a graph embedding into $\fine(S_1).$ Top left: a graph $G$ on 6 vertices that is inadmissible as a subgraph of $\fine(S_1)$. Without the central 
(unlabeled) vertex, this graph is inadmissible as a subgraph of $\fine(S_{0,2}.$ Bottom left, top right, and bottom right: realizing $v_1,\ldots,v_4$ as curves on the torus. Bottom right: $v_5$ must be contained in the unshaded regions and be disjoint from $v_1$ and $v_4$, which is not possible.}\label{fig:combo:inadmissiblegraph1}
\end{center}
\end{figure}
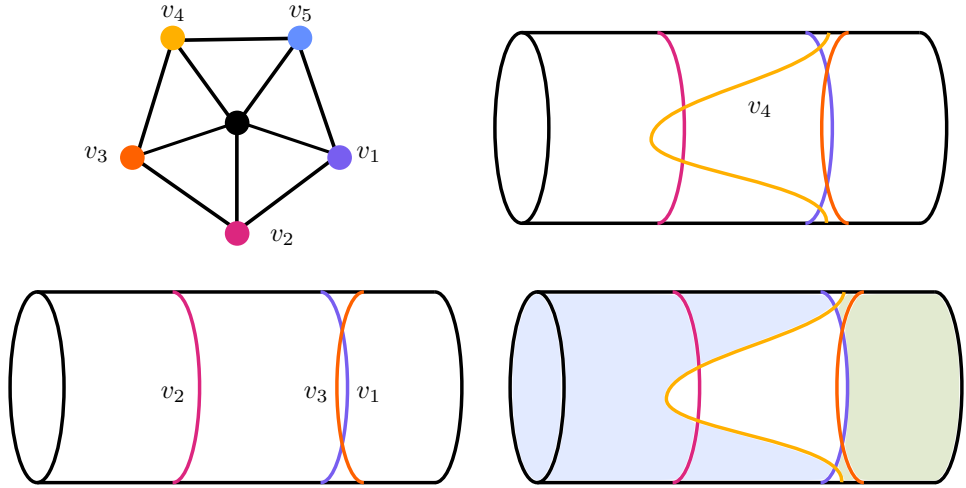

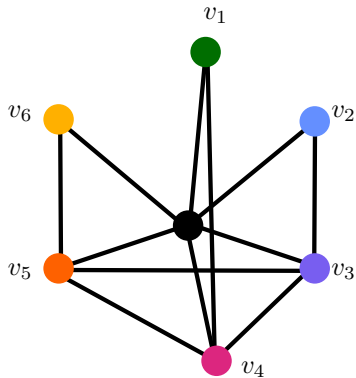
\begin{figure}[h]
\begin{center}
\begin{tikzpicture}
    \node[anchor = south west, inner sep = 0] at (0,0) {\includegraphics[width=1.5in]{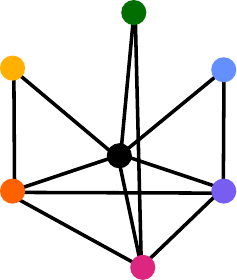}};
    %\draw[help lines] (0,0) grid (6,6);
    \node at (2.3,4.8) {$v_1$};
    \node at (2.8,0.1) {$v_4$};
    \node at (4,3.5) {$v_2$};
    \node at (4, 1.4) {$v_3$};
    \node at (-0.3,1.4) {$v_5$};
    \node at (-0.3,3.5) {$v_6$};
    \end{tikzpicture}
\caption{A graph on 7 vertices that is inadmissible as a subgraph of $\fine(S_1)$. Without the central (unlabeled) vertex, the graph is inadmissible as a subgraph of $\fine(S_{0,2}).$}\label{fig:combo:inadmissiblegraph2}
\end{center}
\end{figure}

\begin{proof}
    \pit{Inadmissibility of $G_1$} Consider the graph $G_1$ defined by a $(2n-1)$-cycle with vertices $v_1,\ldots,v_{2n-1}$, where each $v_i$ is adjacent to $v_{i+1}$ and $v_1$ is adjacent to $v_{n-1}$. We prove $G_1$ is inadmissible in $\fine(S_{0,2}).$

Suppose $G_1$ were admissible and let $\phi:G\to \fine(S_1)$ be the embedding. We will restrict the images of the vertices in the cycle. First, $\phi(v_1)$ and $\phi(v_2)$ are disjoint cores of the annulus, and without loss of generality, $\phi(v_1)$ is on the right while $\phi(v_2)$ is on the left. By Lemma~\ref{combo:lemma:niceadmissiblesubgraphs}, the images of the remainder of the vertices will be both to the left and to the right of any preexisting curve they intersect. Thus, $\phi(v_3)$ is to the right of $\phi(v_2)$ but is also both to the left and to the right of $\phi(v_1).$ As we continue, we notice that $\phi(v_4)$ must be to the left of $\phi(v_3)$ in order to intersect $\phi(v_2).$ Generalizing this, we notice that for all $i,$ $\phi(v_{2i})$ must be to the left of $\phi(v_{2i-1}),$ while $\phi(v_{2i+1})$ must be to the right of $\phi(v_{2i}).$

Therefore we have that $\phi(v_{2n-1})$ must be to the right of $\phi(v_{2n-2})$ and to the left of $\phi(v_1)$ (in order to intersect $\phi(v_2)$). However, $\phi(v_1)$ and $\phi(v_{2n-1})$ intersect, so there can be no curve between them. Thus, there is no viable image of $v_{2n-1}.$ 

\pit{Inadmissibility of $G_2$} The key to showing that this graph structure is inadmissible is the 4-clique with vertices (all of which intersect) dangling off of it. 

We first notice that, up to graph automorphism, $v_3,$ $v_4,$ and $v_5$ are equivalent (and the curves must be disjoint). Without loss of generality, $v_4$ is positioned between $v_3$ and $v_5.$ Since $v_1$ is disjoint from $v_4,$ it must be on exactly one side of $v_4.$ However, $v_1$ must intersect both $v_3$ and $v_5$ and therefore must intersect both sides of the annulus, and therefore would have to intersect $v_4,$ a contradiction. 

(We note that although we did not mention $v_2$ and $v_6$ directly, their existence was imperative since the vertices in the 3-cycle could have otherwise been reordered to make the graph admissible.)
\end{proof}

\begin{lemma}[Inadmissible graphs, the torus case]\label{prop:combo:inadmissibletorus}
    The following two graphs (or families) are inadmissible as subgraphs of $\fine(S_1).$
    \begin{enumerate}
        \item $G_1 =$ the graph with $2n$ vertices which is a $2n-1$-cycle with a central coned off vertex (such a graph is called a wheel) for $n\geq 3$. An example is pictured in the top left of Figure~\ref{fig:combo:inadmissiblegraph1}.
        \item $G_2=$ the graph with 7 vertices pictures in Figure~\ref{fig:combo:inadmissiblegraph2}.
    \end{enumerate}
\end{lemma}

\begin{proof}
    Both cases follow immediately from noticing that the unlabeled vertices in the figures are cone vertices. Thus all curves in a realization must be disjoint from the curves represented by the central vertices, so we may consider $S_1$ cut along the image of the central vertices. We then obtain an annulus, and by  Lemma~\ref{prop:combo:inadmissibletorus}, we conclude that $G_1$ and $G_2$ are both inadmissible in $\fine(S_1).$
\end{proof}

As shown in Appendix~\ref{appendix:smallgraphsintorus}, $G_1$ of Lemma~\ref{prop:combo:inadmissibletorus} on 6 vertices is actually the smallest graph to be inadmissible in $\fine(S_1).$ We can also show it is the unique inadmissible graph on 6 vertices using Lemma~\ref{combo:lemma:subgraphadmissibility} on most graphs on 6 vertices.

What we actually used in the proof of Lemma~\ref{prop:combo:inadmissibletorus} was that the realization was comprised entirely of curves homotopic to each other in the torus, and, once the torus is cut open along the cone vertex, each curve was separating and homotopic to a boundary component. We therefore have a generalization of the above. Let $\alpha$ be a curve in $\surf$. Define $\fine_\alpha(\surf)$ to be the subgraph of $\fine(\surf)$ induced by all curves isotopic to $\alpha.$

\begin{proposition}\label{combo:prop:inadmissiblesingleisotopy}
    Let $\surf$ be an orientable surface. Then $G=G_{2n},$ the graph on $2n$ vertices which is a $2n-1$-cycle with a central coned off vertex (a wheel) for $n\geq 3$ is inadmissible as a subgraph of $\mathcal{C}_\alpha^\dagger(\surf).$
\end{proposition}
\begin{proof}
    Suppose $G$ were admissible. By cutting $\surf$ along the central (cone) vertex, which we call $\alpha,$ we obtain 2 boundary components, one of which is called the ``left'' and the other the ``right''. Since all other vertices contain one boundary component on each of their sides in $\surf\setminus \alpha,$ we have a well-defined notion of left and right. We then follow the same proof as those of Lemmas~\ref{prop:combo:inadmissibleannulus} and \ref{prop:combo:inadmissibletorus} to obtain a contradiction.
\end{proof}

Proposition~\ref{combo:prop:inadmissiblesingleisotopy} implies that inadmissible graphs in surfaces will be heavily surface-type dependent. In particular, if we can construct a graph inadmissible by one isotopy class and then combine it with multiple others in such a way that implies that isotopy classes will have to be repeated, that will amount to an inadmissible graph. This is the inspiration behind the remainder of this section.

Let $G$ and $H$ be graphs. Define the \emph{join} of $G$ and $H,$ denoted $G\#H,$ to be the graph formed by taking the disjoint union of $G$ and $H$ and adding edges between $g\in V(G)$ and $h\in V(H)$ for all $g$ and $h.$

\begin{lemma}[Inadmissible graphs, the pants case] Let $G$ be a graph inadmissible in $\fine(S_{0,2})$ whose complement is connected. Let $\Gamma = G \# G.$ Then $\Gamma$ is inadmissible in $\fine(S_{0,3}).$
\end{lemma}

\begin{proof}
    Suppose $\Gamma$ were admissible. Take the first copy of $G.$ Since $G$ is inadmissible in $\fine(S_{0,2}),$ vertices of $G$ must represent curves parallel to all three boundary components of $\fine(S_{0,3}).$ Moreover, since the complement of $G$ is connected, any curve disjoint from the curves represented by vertices of $G$ must be even closer to $\partial S_{0,3}.$ 

    Now consider the second copy of $G.$ Call it $G'.$ By the same argument, its vertices must be represented by curves parallel to all three boundary components. Moreover, all of the curves of $G'$ must be disjoint from the curves of the original $G,$ and therefore must be closer to the boundary components of the pair of pants than those of $G.$ This is so because $S_{0,3}\setminus V(G)$ is a collection of 3 annuli (coming from the boundary of $S_{0,3}$) and arbitrarily many disks. However, by the same argument as above, $V(G)$ must be closer to $\partial S_{0,3}$ than $V(G'),$ a contradiction.

    Therefore, $\Gamma$ is inadmissible in $\fine(S_{0,3}).$
\end{proof}

With the above in mind, we will are ready to describe a new construction of a finite graph that does not appear as a subgraph of $\fine(S_{g,b})$.

\begin{proof}[Proof of Theorem~\ref{maintheorem:forbiddensubgraphs}]
    Consider the surface $S_{g,b}.$ Then, a pants decomposition of $S_{g,b}$ has $2g+b-2$ connected components. We will construct a graph $\mathcal{G}$ that, if admissible as a subgraph of $\fine(S_{g,b}),$ would be realized by curves supported on $2g+b-1$ disjoint subsurfaces.

    Let $\Gamma$ be a graph inadmissible in $\fine(S_{0,3})$ (and therefore also in $\fine(S_{0,2})$). Let $v$ be an isolated vertex. Define \[\mathcal{G} = \#_{2g+b-1} (\Gamma \amalg \{v\}).\]

    Since $\Gamma$ is inadmissible in $\fine(S_{0,3}),$ so is $\Gamma \amalg \{v\}.$ In a realization of $\mathcal{G}$, the copies of $\Gamma\amalg \{v\}$ would be supported on at least $2g+b-1$ distinct disjoint subsurfaces. Moreover, the curves in a realization cannot be supported on $2g+b$ distinct disjoint subsurfaces since each $v$ must intersect all curves in a realization of its corresponding $\Gamma.$ 

    However, any decomposition of $S_{g,b}$ into $2g+b-1$ subsurfaces must include at least one annulus or pair of pants, thus making it impossible for $\mathcal{G}$ to be realized in $S_{g,b}.$
\end{proof}

%It is striking to compare known results for inadmissible subgraphs of curve graphs and fine curve graphs of compact surfaces. Both complete graphs and half graphs have $O(g)$ vertices while the best-known bound on inadmissible graphs in fine curve graphs also has $O(g)$ vertices. 
At first glance, one would expect that fine curve graphs have significantly more room to admit all subgraphs beneath a certain size, but the above results point to this not being the case.

\section{Automorphisms of the finitary curve graph}\label{sec:aut}

In this section, we prove Theorem~\ref{maintheorem:aut}. The goal is to reduce this to the main theorem of Booth--Minahan--Shapiro \cite{BMS}. In particular, we will show that every automorphism $\psi$ of $\finfine(S_g)$ induces an automorphism of $\onefine(S_g)$ by showing that $\psi$ preserves the set of edges corresponding to 0 or 1 points of intersection. 

\begin{proposition}\label{prop:preserved01edges}
    Suppose $u$ and $v$ are a pair of curves adjacent in $\finfine(S_g)$ with $|u\cap v|\leq 1$ and $\psi\in\Aut\finfine(S_g).$ Then $|\psi(u)\cap \psi(v)|\leq 1.$
\end{proposition}

To prove this proposition, we make the following key observation: suppose a curve $a$ is contained in a finite union of curves $\bigcup u_i.$ Then, if a curve $b$ intersects $a$ infinitely many times, it intersects at least one of the $u_i$ infinitely many times. As it turns out, the converse is also true, and we formalize this observation as the following lemma. 

Define the \textit{link of a vertex} $v$, denoted $\link(v),$ to be the induced subgraph of all vertices adjacent to $v.$ Define the \textit{link of a set of vertices} $\{v_1,\ldots,v_k\},$ denoted $\link(v_1,\ldots,v_k),$ to be $\bigcap_{i=1}^k\link(v_i).$

\begin{lemma}\label{lemma:containment}
    Let $v_1,\ldots,v_k$ be a set of vertices in $\finfine(S_g).$ Then a curve $u\subset \cup v_i$ if and only if $\link(v_1,\ldots,v_k)\subseteq \link(u).$ 
\end{lemma}

\begin{proof}
    $\Longrightarrow$ This follows from the observation above: if a curve intersects each of the $v_i$ finitely many times, it must therefore intersect $u$ finitely many times.

    % $\Longleftarrow$ We will prove this by contrapositive. Suppose $u\not\subset \bigcup v_i.$ Since $\bigcup v_i$ is compact, $u\setminus \bigcup v_i$ is open and nonempty in $u$ and therefore contains an interval of $u.$ By the work in the proof of Theorem~\ref{maintheorem.countable.subgraphs.in.finitary}, there is a curve $w$ that intersects all of the $v_i$ finitely many times and $u$ infinitely many times. Thus $w\in \link(v_1,\ldots,v_k)$ but $w\not\in\link(u).$
    $\Longleftarrow$ We will prove this by contrapositive. Suppose $u\not\subset \bigcup v_i.$ We will find a curve $\alpha \in \link(v_1,\ldots,v_k)$ such that $\alpha\not\in \link(u).$ 
    
    Since $\bigcup v_i$ is compact, $u\setminus \bigcup v_i$ is open and nonempty in $u$ and therefore contains an open interval $O$ of $u.$ Let $x,y\in O$ and consider $S_g\setminus u.$ Then, there is an arc in $S_g\setminus u$ connecting $x$ and $y.$ Applying Corollary~\ref{cor:choosecurve}, there is an arc $a$ whose endpoints are in $O$ and that intersects all $v_i$ finitely many times. Glue the surface back up along $u$ via the identity and connect the endpoints of $a$ along $O$ to form a curve $\alpha.$ As such, $|\alpha \cap u|=\infty$ but $|\alpha\cup v_i| = |a\cup v_i|<\infty,$ meaning $\alpha \in \link(v_1,\ldots,v_k)$ but $\alpha\not\in \link(u).$
\end{proof}

The second lemma we need is the following.

\begin{lemma}\label{lemma:combinatorial}
    Suppose $u$ and $v$ are a pair of adjacent vertices in $\finfine(S_g).$ Then $|u\cap v|\leq 1$ if and only if there are no essential simple closed curves in $u\cup v$ other than $u$ and $v.$
\end{lemma}

\begin{proof}
    \p{Suppose $\mathbf{|u\cap v|\leq 1}$} If $u$ and $v$ are disjoint, there are no curves other than $u$ and $v$ in their union. If they intersect at one point, any curve in their union (other than $u$ and $v$) must contain a $p\in u\setminus v.$ We must then follow $u$ until the intersection $u\cap v,$ upon which we must follow the entirety of $v$ until the intersection of $u\cap v$ again---a contradiction since the constructed curve must be simple, and any such construction must self-intersect at $u\cap v.$

    \p{Suppose $\mathbf{|u\cap v|\geq2}$} We will show that there is an essential simple closed curve other than $u$ and $v$ in $u\cup v$. Let $x_1\neq x_2$ be consecutive points of intersection between $u$ and $v$ when considered along $u.$ (This can be made precise by parameterizing $u$ as $u:I\to S$ with $u(0)=u(1)\not\in u\cap v$ and choosing $x_1=u(t_1)$ and $x_2=u(t_2)$ to be such $u(t)\not\in u\cap v$ for all $t_1<t<t_2.$) Then, there is an arc of $u$, which we call $u'$, with $x_1$ and $x_2$ as endpoints, so that $u'\cap v=\{x_1,x_2\}.$ Moreover, $v\setminus\{x_1,x_2\}$ is two arcs, $v_1$ and $v_2.$ 

    We claim that $u'\cup v_1$ or $u'\cup v_2$ must be essential. If both are inessential, it is to be the case that $v$ is inessential, a contradiction. Call the essential curve $w.$

    Then $w\subset u\cup v,$ as desired.
\end{proof}

\begin{proof}[Proof of Proposition~\ref{prop:preserved01edges}]
    The proposition now follows from Lemma~\ref{lemma:combinatorial}.
\end{proof}

\begin{proof}[Proof of Theorem~\ref{maintheorem:aut}]
    By Proposition~\ref{prop:preserved01edges}, edges corresponding to disjointness or 1 point of intersection are preserved, so any automorphism of $\finfine(S_g)$ induces an automorphism of the fine 1-curve graph $\onefine(S_g).$ Let $\Psi_1:\Aut\finfine(S_g)\to \Aut\onefine(S_g)$ be the map such that $\Psi_1(f)=\overline{f}$ is the automorphism induced by $f$. We note that $f$ and $\overline{f}$ act the same way on the vertex sets of their corresponding graphs. Let $\Psi_2:\Aut\onefine(S_g)\to \Homeo(S_g)$ be the map from Booth--Minahan--Shapiro~\cite{BMS}. 

    Let $\Phi:\Homeo(S_g)\to \Aut \finfine(S_g)$ be the natural map. We claim that $\Psi_2\circ \Psi_1=\Phi^{-1}.$

    Let $\varphi\in \Homeo(S_g)$. Then,
    \begin{align*}
        \Psi_2\circ \Psi_1\circ \Phi(\varphi) &= \Psi_2\circ \Psi_1(f_{\varphi}), \textrm{ where } f_{\varphi} \textrm{ permutes vertices as prescribed by }\varphi\\
        &= \Psi_2(\overline{f_{\varphi}})\\
        &= \Psi_2(\Psi_2^{-1}(\varphi))\\
        &= \varphi.
    \end{align*}
    Conversely, let $f\in \Aut \onefine(S_g).$ Then,
    \begin{align*}
        \Phi\circ\Psi_2\circ \Psi_1(f) &= \Phi \circ \Psi_2(\overline{f})\\
        &=\Phi(\varphi_{\overline{f}}), \textrm{ where } \varphi_{\overline{f}} \textrm{ permutes curves as prescribed by } \overline{f}\\
        &=\Phi (\varphi_f), \textrm{ since } f \textrm{ and } \overline{f} \textrm{ permute vertices in the same way}\\
        &=f.
    \end{align*}

We conclude that the natural map $\Phi: \Homeo(S_g)\to \Aut \finfine(S_g)$ is an isomorphism. 
\end{proof}

We connect this back to the Erd\H{o}s-R\'enyi graph. Although $\finfine(S_g)$ has the Erd\H{o}s-R\'enyi graph as an induced subgraph, it does not have certain qualities that the Erd\H{o}s-R\'enyi graph possesses. In particular, because of Property ($*$) from Section~\ref{sec:subgraphs}, the Erd\H{o}s-R\'enyi graph is \textit{highly symmetric}; that is, any isomorphism of induced subgraphs extends to an automorphism of the entire graph. Our Theorem~\ref{maintheorem:aut} implies that automorphisms of $\finfine(S_g)$ are extremely rigid and preserve many topological properties. Not only is $\finfine(S_g)$ not highly symmetric, but for any nontrivial subgraph of $\finfine(S)$, including single vertices, there exists an isomorphism to another induced subgraph of $\finfine(S_g)$ that cannot be extended to an automorphism of $\finfine(S_g).$ 

\appendix
\section{All graphs on $\leq 5$ vertices are admissible in $\fine(S_1)$}\label{appendix:smallgraphsintorus}

We include this section for completeness. We show via casework that every graph on 5 vertices is admissible as a subgraph of $\fine(S_1).$ We will begin casework by considering the size of the largest clique.

\p{5-clique} This is realizable as 5 parallel disjoint curves.

\p{4-clique} We perform casework on how many of the 4 curves in the clique the final curve, $\gamma,$ must intersect. We begin by drawing the 4 curves that comprise the clique.
\begin{enumerate}
    \item \textit{1 intersection.} Draw the first 4 parallel disjoint curves in any order. Draw $\gamma$ close to the one curve it must intersect. (An alternative proof method is to note that if it is known that all graphs on 4 vertices are admissible, then this graph is also admissible by (4) of Lemma~\ref{combo:lemma:subgraphadmissibility}.)
    \item \textit{2 intersections.} When drawing the 4 parallel disjoint curves, draw the two curves consecutively. Then, draw $\gamma$ between the two curves it must intersect and isotope it to intersect both. (An alternative proof method is to note that if it is known that all graphs on 4 vertices are admissible, then this graph is also admissible by (4) of Lemma~\ref{combo:lemma:subgraphadmissibility}.)
    \item \textit{3 intersections.} This can be accomplished by (2) of Lemma~\ref{combo:lemma:subgraphadmissibility}.
    \item \textit{4 intersections.} This can be accomplished by (1) of Lemma~\ref{combo:lemma:subgraphadmissibility}.
\end{enumerate}

\p{3-clique} We will do this by casework, looking at the length of the largest cycle.
\begin{enumerate}[leftmargin=3\parindent]
    \item[3-cycle.] Fix a 3-cycle in $G$ and let $v_4$ and $v_5$ be the final 2 vertices, which we will attach to the 3-cycle in sequence. In this case, $v_4$ can be attached only via 0 edges or 1 edge (otherwise there sill be a 4-cycle). 

    \textit{0 attachments.} In this case, we can apply (1'), (2'), or both of them of Lemma~\ref{combo:lemma:subgraphadmissibility}, depending on whether $v_5$ has degree 0, 2, or 1, respectively.

    \textit{1 attachment.} We can apply (2) or (2') of Lemma~\ref{combo:lemma:subgraphadmissibility} if the degree of $v_5$ is 1 or (1) if the degree of $v_5$ is 0. The new case is degree 2, in which case we can apply (2) to the 3-clique and then (4), as $v_4,v_5$ would be a blowup of $v_4$ into a clique.
    
    \item[4-cycle.] We do casework on the degree of the final vertex.
    \begin{enumerate}
        \item[Degree 0] This is possible by (1) of Lemma~\ref{combo:lemma:subgraphadmissibility}.
        \item[Degree 1] This is possible by (2) of Lemma~\ref{combo:lemma:subgraphadmissibility}.
        \item[Degree 2] (and higher degrees) This is impossible since then there would be a 5-cycle (since there must also be a 3-clique).
    \end{enumerate}
    
    \item[5-cycle.] In this case, we begin with a 5-cycle with 1 additional edge to account for the existence of a triangle. We do casework on the number of additional edges.
    \begin{enumerate}
        \item[0.] This is realizable as in Figure~\ref{fig:combo:5cycle1triangle}.

        \begin{figure}[h]
        \begin{center}
        \begin{tikzpicture}
            \node[anchor = south west, inner sep = 0] at (0,0) {\includegraphics[width=4in]{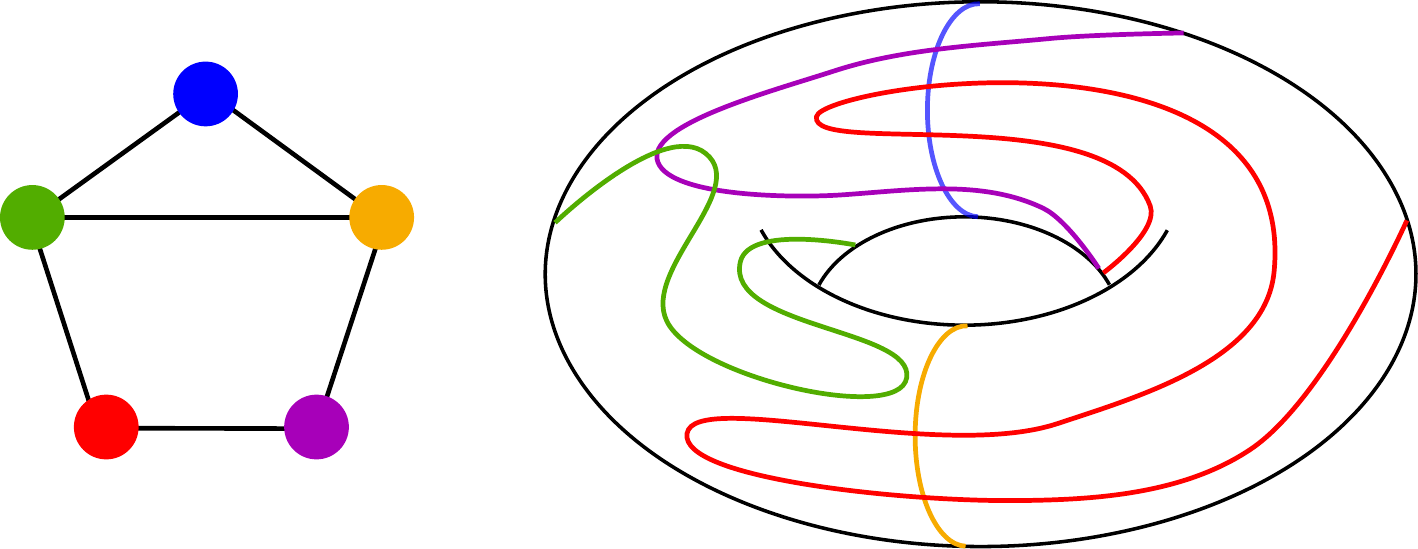}};
            %\draw[help lines] (0,0) grid (10,4);
            \node at (1.5,3.7) {$v_1$};
            \node at (3.25,2.5) {$v_2$};
            \node at (2.65,0.6) {$v_3$};
            \node at (0.3,0.6) {$v_4$};
            \node at (-0.3,2.5) {$v_5$};
        
            \node at (6.9,4.25) {$v_1$};
            \node at (7,1.3) {$v_2$};
            \node at (8.5,3.9) {$v_3$};
            \node at (9.5,2) {$v_4$};
            \node at (4.5,2) {$v_5$};
            \end{tikzpicture}
        \caption{A 5-cycle with one additional edge as an admissible subgraph of $\fine(S_1)$.}\label{fig:combo:5cycle1triangle}
        \end{center}
        \end{figure}
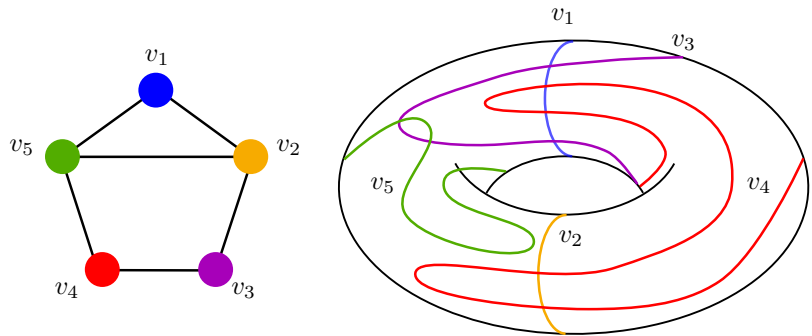

        \item[1.] There are 2 options; both are realizable as in Figure~\ref{fig:combo:5cycle2edges}, one on top and one on the bottom.

        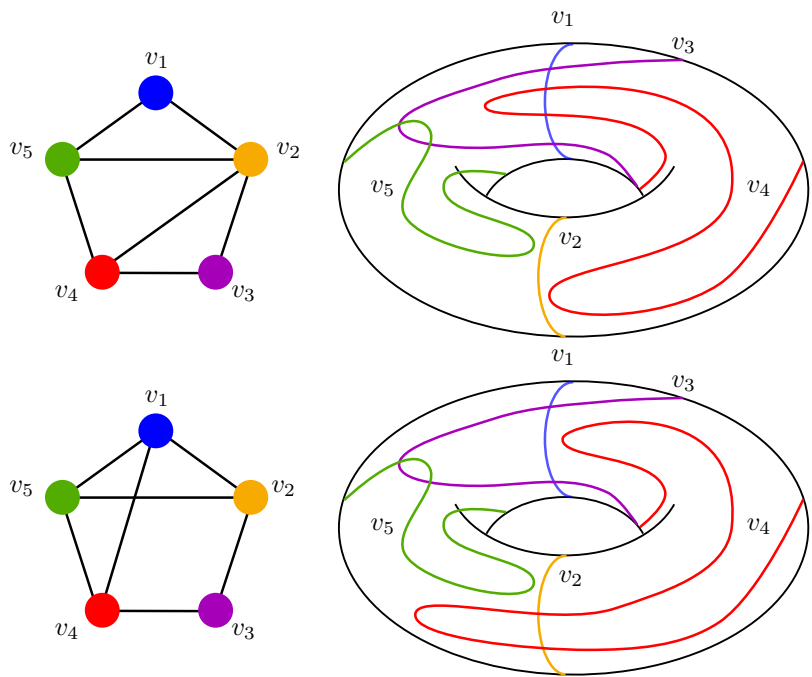
\begin{figure}[h]
        \begin{center}
        \begin{tikzpicture}
            \node[anchor = south west, inner sep = 0] at (0,0) {\includegraphics[width=4in]{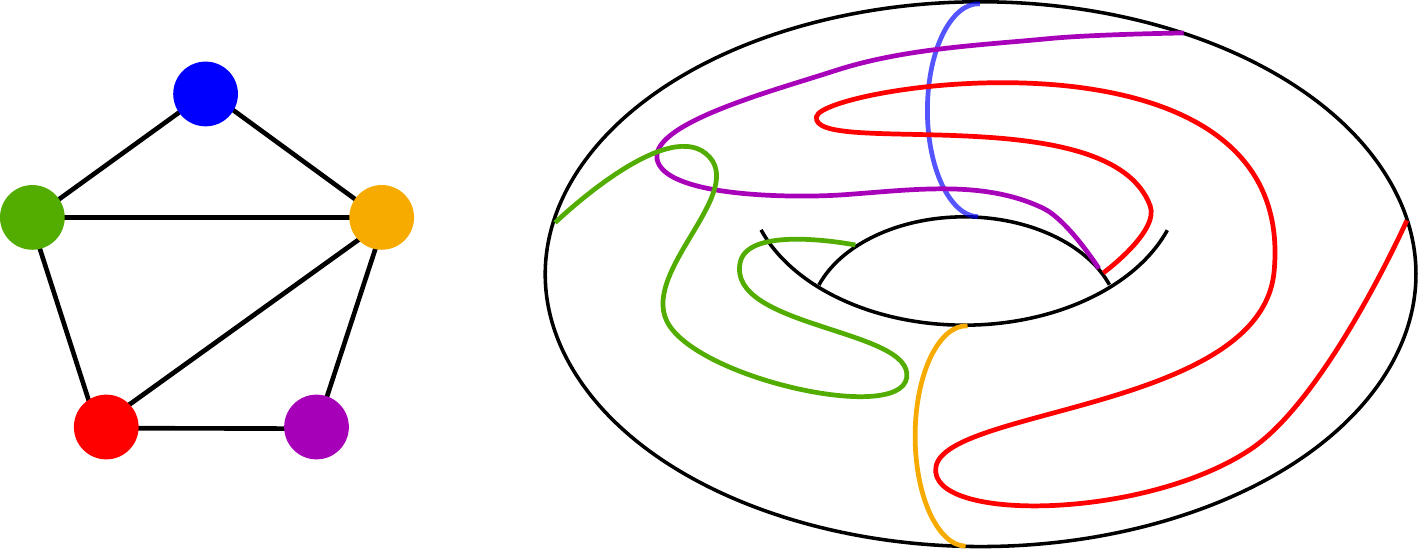}};
            %\draw[help lines] (0,0) grid (10,4);
            \node at (1.5,3.7) {$v_1$};
            \node at (3.25,2.5) {$v_2$};
            \node at (2.65,0.6) {$v_3$};
            \node at (0.3,0.6) {$v_4$};
            \node at (-0.3,2.5) {$v_5$};
        
            \node at (6.9,4.25) {$v_1$};
            \node at (7,1.3) {$v_2$};
            \node at (8.5,3.9) {$v_3$};
            \node at (9.5,2) {$v_4$};
            \node at (4.5,2) {$v_5$};
            \end{tikzpicture}
            \begin{tikzpicture}
            \node[anchor = south west, inner sep = 0] at (0,0) {\includegraphics[width=4in]{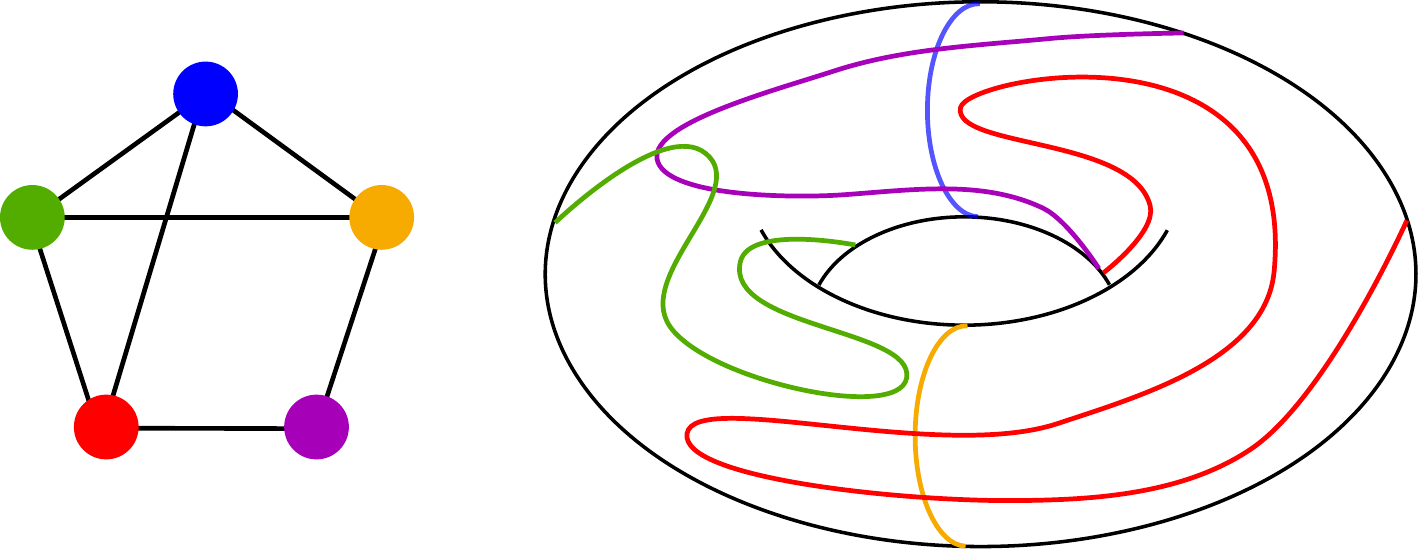}};
            %\draw[help lines] (0,0) grid (10,4);
            \node at (1.5,3.7) {$v_1$};
            \node at (3.18,2.5) {$v_2$};
            \node at (2.65,0.6) {$v_3$};
            \node at (0.3,0.6) {$v_4$};
            \node at (-0.3,2.5) {$v_5$};

            \node at (6.9,4.25) {$v_1$};
            \node at (7,1.3) {$v_2$};
            \node at (8.5,3.9) {$v_3$};
            \node at (9.5,2) {$v_4$};
            \node at (4.5,2) {$v_5$};

            \end{tikzpicture}
        \caption{Two 5-cycles with 2 additional edge as admissible subgraphs of $\fine(S_1)$.}\label{fig:combo:5cycle2edges}
        \end{center}
        \end{figure}
        
        \item[2.] There is 1 option; it is realizable as in Figure~\ref{fig:combo:5cycle3edges}.
        \begin{figure}[h]
        \begin{center}
        \begin{tikzpicture}
            \node[anchor = south west, inner sep = 0] at (0,0) {\includegraphics[width=4in]{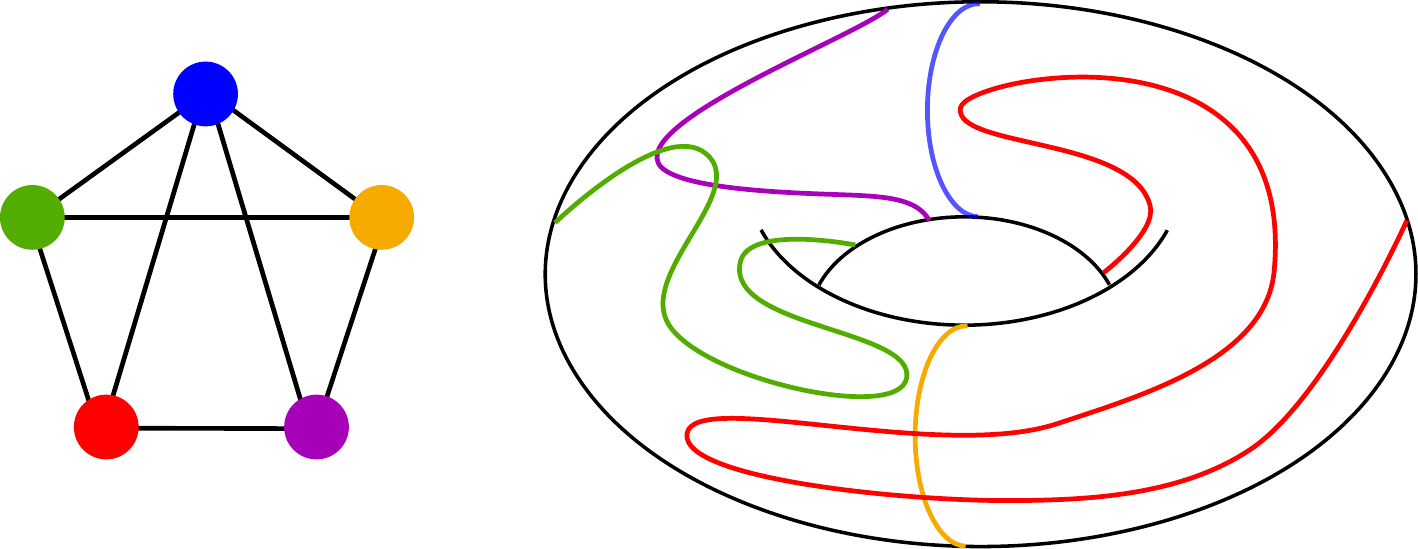}};
            %\draw[help lines] (0,0) grid (10,4);
            \node at (1.5,3.7) {$v_1$};
            \node at (3.25,2.5) {$v_2$};
            \node at (2.65,0.6) {$v_3$};
            \node at (0.3,0.6) {$v_4$};
            \node at (-0.3,2.5) {$v_5$};
        
            \node at (6.9,4.25) {$v_1$};
            \node at (7,1.3) {$v_2$};
            \node at (8.5,3.9) {$v_3$};
            \node at (9.5,2) {$v_4$};
            \node at (4.5,2) {$v_5$};
            \end{tikzpicture}
        \caption{A 5-cycle with three additional edges as an admissible subgraph of $\fine(S_1)$.}\label{fig:combo:5cycle3edges}
        \end{center}
        \end{figure}
        \item[3.] We cannot add 3 or more edges since then we will have a 3-clique.
    \end{enumerate}
 \end{enumerate}

\p{2-clique} We look at the length of the largest cycle.
\begin{enumerate}[leftmargin=4\parindent]
    \item[No cycles.] This is a tree, which is admissible by (2') of Lemma~\ref{combo:lemma:subgraphadmissibility}.
    \item[2- or 3-cycle.] This is not possible on account of not existing and a 3-cycle being a 3-clique, respectively.
    \item[4-cycle.] We now do casework on the degree of the final vertex, $\gamma$.
    \begin{enumerate}
        \item[Degree 0] This is by (1) of Lemma~\ref{combo:lemma:subgraphadmissibility}.
        \item[Degree 1] This is by (2) of Lemma~\ref{combo:lemma:subgraphadmissibility}.
        \item[Degree 2] Now there are 2 cases: either the neighbors of $\gamma$ are adjacent or they are not. If they are adjacent, then there is a triangle, a contradiction. Otherwise, the graph is admissible as in Figure~\ref{fig:combo:4cycle}.

        \begin{figure}[h]
        \begin{center}
        \begin{tikzpicture}
            \node[anchor = south west, inner sep = 0] at (0,0) {\includegraphics[width=4in]{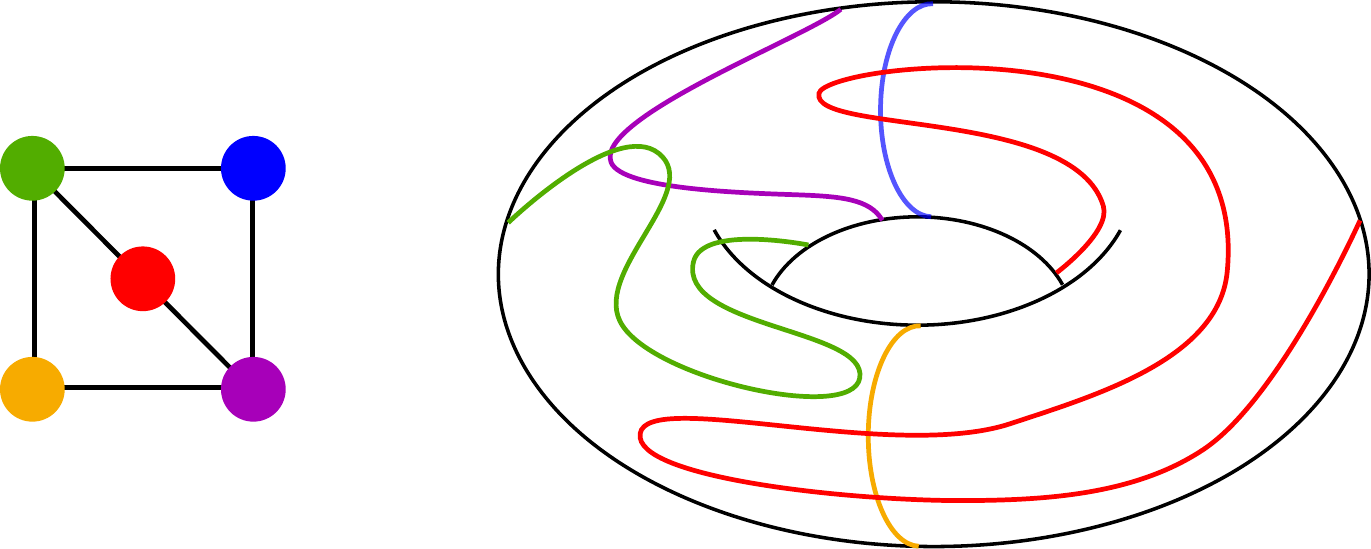}};
            %\draw[help lines] (0,0) grid (10,4);
            \node at (2.15,3.25) {$v_1$};
            \node at (0,0.75) {$v_2$};
            \node at (2.15,0.75) {$v_3$};
            \node at (1.35,2.35) {$v_4$};
            \node at (0,3.25) {$v_5$};
        
            \node at (6.9,4.25) {$v_1$};
            \node at (6.85,1.4) {$v_2$};
            \node at (6,4.25) {$v_3$};
            \node at (9.5,2) {$v_4$};
            \node at (4.34,2) {$v_5$};
            \end{tikzpicture}
        \caption{A 4-cycle with two additional edge as an admissible subgraph of $\fine(S_1)$.}\label{fig:combo:4cycle}
        \end{center}
        \end{figure}
    \end{enumerate}
    \item[5-cycle.] This is possible as in Figure~\ref{fig:combo:5cycle}.
    \begin{figure}[h]
        \begin{center}
        \begin{tikzpicture}
            \node[anchor = south west, inner sep = 0] at (0,0) {\includegraphics[width=4in]{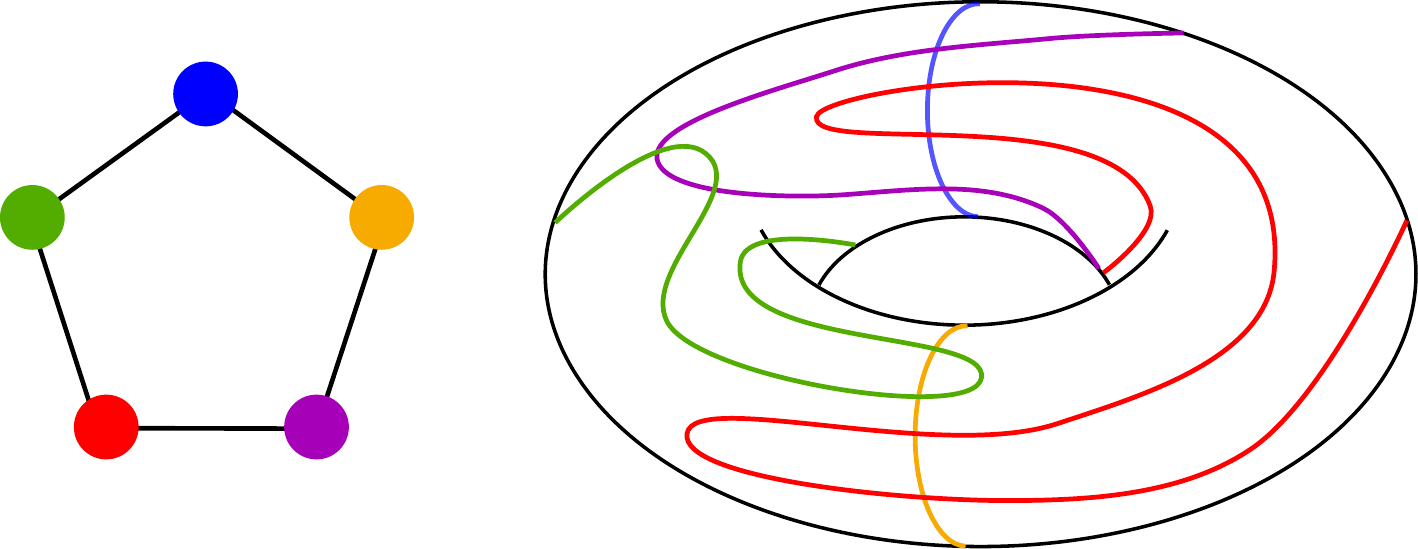}};
            %\draw[help lines] (0,0) grid (10,4);
            \node at (1.5,3.7) {$v_1$};
            \node at (3.25,2.5) {$v_2$};
            \node at (2.65,0.6) {$v_3$};
            \node at (0.3,0.6) {$v_4$};
            \node at (-0.3,2.5) {$v_5$};
        
            \node at (6.9,4.25) {$v_1$};
            \node at (6.9,1.8) {$v_2$};
            \node at (8.5,3.9) {$v_3$};
            \node at (9.5,2) {$v_4$};
            \node at (4.5,2) {$v_5$};
            \end{tikzpicture}
        \caption{A 5-cycle as an admissible subgraph of $\fine(S_1)$.}\label{fig:combo:5cycle}
        \end{center}
        \end{figure}
\end{enumerate}

\p{1-clique} This can be accomplished by choosing 5 curves that all mutually intersect essentially. (This can also be done by (1') of Lemma~\ref{combo:lemma:subgraphadmissibility}.)

\printbibliography

\end{document}